\documentclass[aos,preprint,a4paper,12pt]{imsart}

\usepackage[colorlinks,citecolor=blue,urlcolor=blue,linkcolor=RawSienna,hypertexnames=true]{hyperref}
\usepackage{amsthm,mathtools,thmtools}
\usepackage[capitalize]{cleveref}

\RequirePackage[OT1]{fontenc}

\usepackage[utf8]{inputenc}
\usepackage{amsmath, amssymb,bm, cases, mathtools, thmtools}
\usepackage{verbatim}
\usepackage{graphicx}\graphicspath{{figures/}}
\usepackage{multicol}
\usepackage{tabularx}
\usepackage[usenames,dvipsnames]{xcolor}
\usepackage{mathrsfs}
\usepackage{url}
\usepackage[normalem]{ulem}
\usepackage{xstring}
\usepackage{enumitem}

\input{bib.tex}

\usepackage[colorlinks,citecolor=blue,urlcolor=blue,linkcolor=RawSienna]{hyperref}
\usepackage{hypernat}
\usepackage{datetime}

\DeclareMathAlphabet\EuRoman{U}{eur}{m}{n}
\SetMathAlphabet\EuRoman{bold}{U}{eur}{b}{n}

\usepackage[capitalize]{cleveref}

\crefname{assumption}{Assumption}{Assumptions}
\crefname{claim}{Claim}{Claims}

\makeatletter
\let\reftagform@=\tagform@
\def\tagform@#1{\maketag@@@{\ignorespaces\textcolor{gray}{(#1)}\unskip\@@italiccorr}}
\renewcommand{\eqref}[1]{\textup{\reftagform@{\ref{#1}}}}
\makeatother

\newcommand{\LATER}[1]{\error}
\newcommand{\fLATER}[1]{\error}
\newcommand{\TBD}[1]{\error}
\newcommand{\fTBD}[1]{}
\newcommand{\PROBLEM}[1]{\error}
\newcommand{\fPROBLEM}[1]{\error}
 
\newcommand{\COMMENTINGMSG}{}

\declaretheorem[style=plain,numberwithin=section,name=Theorem]{theorem}
\declaretheorem[style=plain,sibling=theorem,name=Lemma]{lemma}

\declaretheorem[style=plain,sibling=theorem,name=Corollary]{corollary}
\declaretheorem[style=plain,sibling=theorem,name=Claim]{claim}

\declaretheorem[style=definition,sibling=theorem,name=Definition]{definition}

\declaretheorem[style=definition,sibling=theorem,name=Example]{example}
\declaretheorem[style=remark,sibling=theorem,name=Remark]{remark}

\declaretheorem[style=definition,name=Condition]{conditionINNER}
\newenvironment{condition}[1]
 {\renewcommand{\theconditionINNER}{#1}\conditionINNER}
 {\endconditionINNER}
\crefformat{conditionINNER}{#2(#1)#3}
\crefmultiformat{conditionINNER}
  {(#2#1#3)}
  { and~(#2#1#3)}
  {, (#2#1#3)}
  { and~(#2#1#3)}
\Crefformat{conditionINNER}{Condition~#2(#1)#3}
\Crefmultiformat{conditionINNER}
  {Conditions~(#2#1#3)}
  { and~(#2#1#3)}
  {, (#2#1#3)}
  { and~(#2#1#3)}

\declaretheoremstyle[
    spaceabove=-6pt, 
    spacebelow=6pt, 
    headfont=\normalfont\bfseries, 
    bodyfont = \normalfont,
    postheadspace=1em, 
    qed=$\square$, 
    headpunct={{}}]{myproofstyle}

\numberwithin{equation}{section}
\numberwithin{theorem}{section}

\renewcommand{\appendixname}{}
\usepackage{amssymb}
\usepackage{mathrsfs}
\usepackage{leftidx}

\renewenvironment{quote}
               {\list{}{\rightmargin\leftmargin}%
                \item\relax}
               {\endlist}

\usepackage{mathptmx}
\usepackage[total={210mm,297mm},
top=45mm, bottom=35mm, left=35mm, right=35mm]{geometry}

\arxiv{arXiv:1612.09305}

\startlocaldefs

\def\[#1\]{\begin{align}#1\end{align}}
\def\*[#1\]{\begin{align*}#1\end{align*}}

\newcommand{\argdot}{\cdot}

\newcommand{\st}{\,:\,}

\newcommand{\Naturals}{\mathbb{N}}

\newcommand{\Reals}{\mathbb{R}}
\newcommand{\Nats}{\mathbb{N}}
\newcommand{\Ints}{\mathbb{Z}}

\newcommand{\NNReals}{\Reals_{\ge 0}}
\newcommand{\PosReals}{\Reals_{> 0}}

\newcommand{\conv}{\textrm{conv}}

\newcommand{\dee}{\mathrm{d}}

\DeclareMathOperator*{\newlim}{\mathrm{lim}\vphantom{\mathrm{infsup}}}
\DeclareMathOperator*{\newmin}{\mathrm{min}\vphantom{\mathrm{infsup}}}

\DeclareMathOperator*{\newinf}{\mathrm{inf}\vphantom{\mathrm{infsup}}}
\DeclareMathOperator*{\newsup}{\mathrm{sup}\vphantom{\mathrm{infsup}}}
\renewcommand{\lim}{\newlim}
\renewcommand{\min}{\newmin}

\renewcommand{\inf}{\newinf}
\renewcommand{\sup}{\newsup}

\newcommand{\defn}[1]{\emph{#1}}

\newcommand{\cF}{\mathcal F}
\newcommand{\cG}{\mathcal G}

\DeclareMathOperator*{\Bernoulli}{Bernoulli}

\newcommand{\BorelSets}[1]{\mathcal{B}[#1]}
\newcommand{\NSE}[1]{{^{*}#1}}
\newcommand{\ST}{\mathsf{st}}
\newcommand{\SP}[1]{{^{\circ}#1}}

\newcommand{\InternalSubsets}[1]{\mathcal I[#1]}

\newcommand{\AS}{\mathbb{A}}

\newcommand{\PowerSet}[1]{\mathscr{P}(#1)}
\newcommand{\HReals}{\NSE{\Reals}}

\newcommand{\Model}{P}

\newcommand{\Loss}{\ell}

\newcommand{\NS}[1]{\mathrm{NS}(#1)}

\newcommand{\Risk}{r}

\newcommand{\cS}{\mathcal{S}}

\newcommand{\cA}{\mathcal{A}}
\newcommand{\bV}{\mathbb{V}}

\newcommand{\cC}{\mathcal{C}}

\newcommand{\IFuncs}[2]{\mathrm{I}({#1}^{#2})}

\newcommand{\FiniteSubsets}[1]{{#1}^{[< \infty]}}

\newcommand{\vzero}{\mathbf{0}}
\newtheorem{open problem}{Open Problem}

\newcommand{\Loeb}[1]{\overline{#1}}
\newcommand{\LoebAlgebraX}[2]{\overline{#1}_{#2}}
\newcommand{\LoebAlgebra}[1]{\overline{#1}} %

\newcommand{\ProbMeasures}[1]{\mathcal{M}_1(#1)}

\newcommand{\gX}{X}
\newcommand{\gY}{Y}

\newcommand{\interior}[1]{%
  {\kern0pt#1}^{\mathrm{o}}%
}

\newcommand{\Language}{\mathcal L}

\newcommand{\FRRE}{\mathcal D}
\newcommand{\FRREFS}{\FRRE_{0,FC}}
\newcommand{\sFRREFS}{\NSE{\!\FRREFS}}
\newcommand{\FiniteRiskEstimators}{\FRRE_0}
\newcommand{\refproof}[1]{See \cref{#1} for \IfSubStr{#1}{,}{proofs}{a proof}. }

\newcommand{\sconv}{\NSE{\!\conv}}

\newcommand{\TTheta}{T_{\Theta}}
\newcommand{\ttheta}{t}
\newcommand{\STK}{J_{\Theta}}
\newcommand{\DRiskSpace}{\IFuncs{\HReals}{\STK}}

\newcommand{\ExtFiniteRiskEstimators}{{\FiniteRiskEstimators}^{\!\!\!*}}
\newcommand{\ExtFRRE}{\FRRE^{*}}

\newcommand{\FSE}{(\ExtFiniteRiskEstimators)_{FC}}

\newcommand{\sFiniteRiskEstimators}{\NSE{\!\FiniteRiskEstimators}}
\newcommand{\sFRRE}{\NSE{\!\FRRE}}

\newcommand{\sdelta}{\NSE{\!\delta}}
\newcommand{\sRisk}{\NSE{\!\Risk}}

\newcommand{\EEst}[1]{\mathbb{E}(#1)}

\newcommand{\IP}[2]{\langle #1, #2\rangle}

\newcommand{\Fincomb}[1]{(#1)_{FC}}
\newcommand{\spEssCompleteSubclass}[1]{$\SP{}$essentially complete subclass of $#1$}
\newcommand{\spCompleteSubclass}[1]{$\NSE{}$complete subclass of $#1$}

 \newcommand{\EpsDom}[1]{$#1$-dominated}
 \newcommand{\EpsAdmissible}[2]{$#1$-admissible among $#2$}
 \newcommand{\ExtAdmissible}[1]{extended admissible among $#1$}
 \newcommand{\Admissible}[1]{admissible among $#1$}

 \newcommand{\spp}{$\NSE{}$}
 \newcommand{\spDomIn}[2]{\spp{}dominated in $#1/#2$}
 \newcommand{\spDom}[1]{\spp{}dominated on $#1$}
 \newcommand{\spEpsDomIn}[3]{$#1$-\spp{}dominated in $#2/#3$}
 \newcommand{\spEpsDom}[2]{$#1$-\spp{}dominated on $#2$}
 \newcommand{\spEpsAdmissibleIn}[4]{$#1$-\spp{}admissible in $#2/#3$ among $#4$}
 \newcommand{\spEpsAdmissible}[3]{$#1$-\spp{}admissible on $#2$ among $#3$}
 \newcommand{\spExtAdmissible}[2]{\spp{}extended admissible on $#1$ among $#2$}
 \newcommand{\spAdmissibleIn}[3]{\spp{}admissible in $#1/#2$ among $#3$}
 \newcommand{\spAdmissible}[2]{\spp{}admissible on $#1$ among $#2$}

\newcommand{\eBayesP}[3]{$#1$-Bayes under $#2$ among $#3$}
\newcommand{\eBayes}[2]{$#1$-Bayes among $#2$}
\newcommand{\BayesP}[2]{Bayes under $#1$ among $#2$}

\newcommand{\Bayes}[1]{Bayes among $#1$}

\newcommand{\seBayesP}[3]{$#1$-$\NSE{}$Bayes under $#2$ among $#3$}
\newcommand{\seBayes}[2]{$#1$-$\NSE{}$Bayes among $#2$}
\newcommand{\NSBayesP}[2]{nonstandard Bayes under $#1$ among $#2$}
\newcommand{\NSBayes}[1]{nonstandard Bayes among $#1$}

\newcommand{\GENRS}{\cS^{\cC}}

\newcommand{\RS}[1]{\cS^{#1}}

\newcommand{\Ext}[1]{#1^{*}}

\newcommand{\FSC}{\cF}

\newcommand{\Normal}[2]{\mathcal N(#1,#2)}
\newcommand{\Identity}{I}
\newcommand{\bx}{\mathbf{x}}

\newcommand{\ULoeb}[1]{L_u(#1)}

\newcommand{\symdiff}{\bigtriangleup}

\newcommand{\sint}[1]{\sideset{^*\!\!\!\!}{_{#1}}\int}

\newcommand{\GSA}{\mathcal F}
\newcommand{\sGSA}{\NSE{\!\!\mathcal F}}

\newcommand{\makebib}{\bibliography{biblio}}
\endlocaldefs

\begin{document}

\begin{frontmatter}
\title{On Extended Admissible Procedures \\ and their Nonstandard Bayes Risk}%
\runtitle{Nonstandard Bayes}

\begin{aug}
\author{\fnms{Haosui} \snm{Duanmu}\ead[label=e1]{haosui@utstat.toronto.edu}}
\and
\author{\fnms{Daniel~M.} \snm{Roy}\ead[label=e2]{droy@utstat.toronto.edu}}
\runauthor{Duanmu and Roy}

\affiliation{University of Toronto} %

\address{Dept.\ of Statistical Sciences\\
University of Toronto\\
100 St. George St\\
Toronto, ON M5S 3G3\\
\printead{e1}\\
\phantom{E-mail:\ }\printead*{e2}}

\end{aug}

\begin{abstract}

For finite parameter spaces under finite loss, every Bayes procedure derived from a prior with full support is admissible, and every admissible procedure is Bayes. This relationship already breaks down once we move to finite-dimensional Euclidean parameter spaces. Compactness and strong regularity conditions suffice to repair the relationship, but without these conditions, admissible procedures need not be Bayes. Under strong regularity conditions, admissible procedures can be shown to be the limits of Bayes procedures. Under even stricter conditions, they are generalized Bayes, i.e., they minimize the Bayes risk with respect to an improper prior. In both these cases, one must venture beyond the strict confines of Bayesian analysis. Using methods from mathematical logic and nonstandard analysis, we introduce the class of nonstandard Bayes decision procedures---namely, those whose Bayes risk with respect to some prior is within an infinitesimal of the optimal Bayes risk. Among procedures with finite risk functions, we show that a decision procedure is extended admissible if and only if its nonstandard extension is nonstandard Bayes. For problems with continuous risk functions defined on metric parameter spaces, we derive a nonstandard analogue of Blyth's method that can be used to establish the admissibility of a procedure. We also apply the nonstandard theory to derive a purely standard theorem: when risk functions are continuous on a compact Hausdorff parameter space, a procedure is extended admissible if and only if it is Bayes.

\end{abstract}

\begin{keyword}[class=MSC]
\kwd[Primary ]{62C07}
\kwd{62C10}
\kwd{62A01}
\kwd[; secondary ]{28E05}
\end{keyword}

\begin{keyword}
\kwd{decision theory}
\kwd{complete class theorems}
\kwd{nonstandard analysis}
\end{keyword}
\end{frontmatter}

\COMMENTINGMSG{}

\tableofcontents

\section{Introduction}

There is a long line of research, originating in the seminal work of \citet{Wald47a,Wald47},
connecting admissibility and Bayes optimality.
For finite parameters spaces, one can use intuitive geometric arguments to establish that
every admissible decision procedure is Bayes (see, e.g., \citep[][\S2.10 Thm.~1]{Ferguson}).
In the other direction, elementary arguments show that
every procedure that is Bayes with respect to a prior with full support is admissible \citep[][\S2.3 Thms.~2 and 3]{Ferguson}.
This close relationship between admissibility and Bayes optimality already breaks down for finite-dimensional parameter spaces:
Here, admissible decision procedures are, in general, not Bayes procedures (see, e.g., \citet[][\S4]{stein54}).
Under some regularity conditions, however, every admissible procedure is a \defn{limit} of Bayes procedures \citep{Wald49,LeCam55,Brown86},
although, limits of Bayes procedures are in general neither admissible nor Bayes,
as famously demonstrated by \citet{stein54} in the multivariate normal location model (see also \citep{stein61}).
Under more stringent conditions, admissible procedures are \defn{generalized Bayes} \citep{Sacks63,stone67,Brown71,Berger78}, i.e.,
procedures derived from the mechanical---also known as \emph{formal}---application of Bayes rule with respect to \defn{improper} priors.
In both cases, we must leave the strict confines of the Bayesian formalism to re\"establish the link between admissibility and Bayes optimality.
The price paid for abandoning the standard Bayesian framework---namely, nonconglomerability and its side effects, including marginalization paradoxes---is the subject of an extensive literature (see, e.g., \citep{dawid73,seid98,seid84}).

Here we take a different approach, working within the standard Bayesian theory but
carrying out that work in an unusual setting.
In particular, we rely on results in mathematical logic that
establish the existence of \defn{nonstandard} models of the reals satisfying three principles:
\defn{extension}, which associates every ordinary mathematical object with a nonstandard counterpart called its extension;
\defn{transfer}, which permits us to use first order logic to relate standard and nonstandard structures;
and \defn{saturation}, which gives us a powerful mechanism for proving the existence of nonstandard structures defined in terms of finitely satisfiable collections of first order formula.

Informally speaking,
the utility of nonstandard models for statistical decision theory stems from two sources:
First,
every nonstandard model possesses nonstandard reals numbers, including infinitesimal / infinite positive real numbers that are,
respectively, smaller than / larger than any standard positive real number.  Using such numbers, we can, e.g.,
construct uniform probability measures over infinite intervals that contain the entire standard real line,
or
construct probability measures on the positive real line concentrating all their mass on a positive infinitesimal.
As priors, these structures can be used as extreme statements of uncertainty that do not correspond to any standard prior.
Second,
standard real numbers look discrete in a nonstandard model. Indeed, in a suitably saturated model of the reals,
the standard reals are contained within a \defn{hyperfinite} set, i.e., an infinite set that nonetheless possesses all the first order
properties of a standard finite set.

Using nonstandard analysis and probability theory,
we are able to re\"establish the link between admissibility and Bayesian optimality \emph{without} regularity conditions.
In particular, using a separating hyperplane argument in concert with the three principles outline above,
we show that
a standard decision procedure $\delta$
is extended admissible if and only if, for some nonstandard prior, the Bayes risk of its extension $\sdelta$ is within an infinitesimal
of the minimum Bayes risk among all extensions. Such a decision procedure is said to be nonstandard Bayes.
Assuming $\Theta$ is a metric space and risk functions are continuous, we are able to show that a procedure is admissible if its extension is nonstandard Bayes with respect to a prior that assigns sufficient mass to every standard open ball. The result is a nonstandard variant of Blyth's method, but a single nonstandard prior witnesses the admissibility, rather than a sequence.
We also apply our nonstandard theory to give a standard result:
on compact Hausdorff spaces when risk functions are continuous,
a decision procedure is extended admissible if and only if is Bayes.

\subsection{Overview of the paper}

In \cref{sec: standardprelim},
we introduce basic notions and key results in standard statistical decision theory:
domination, admissibility, and its variants;
Bayes optimality; and basic complete class and essentially complete class results.
(Classic treatments can be found in
\citep{Ferguson} and \citep{BG54}, the latter emphasizing the connection with game theory, but restricting itself to finite discrete spaces.
A modern treatment can be found in \citep{LC98}.)
In \cref{sec:prior},
we follow this introduction of basic principles with a summary of the extensive literature on complete class theorems.

In \cref{sec: nonstandardadmis},
we define nonstandard counterparts of admissibility, extended admissibility, and essential completeness,
which we obtain by ignoring infinitesimal violations of the standard notions,
and then give key theorems relating standard and nonstandard notions for standard decision procedures and their nonstandard extensions, respectively.
For readers unfamiliar with nonstandard analysis and probability, we summarize basic notions and key results in \cref{intronsa,sec:inp}.

In \cref{sec: nonstandardbayes}, we define the nonstandard counterpart to Bayes optimality,
which we also obtain by ignoring infinitesimal violations of the standard notion.
Using saturation and a hyperfinite version of the classical separating hyperplane argument
on a hyperfinite discretization of the risk set,
we show that a decision procedure is extended admissible if and only if it its extension is nonstandard Bayes.

In \cref{sec: CCs under compactness and RC},
we apply the nonstandard theory to obtain a standard result:
assuming the parameter space is compact and risk functions are continuous,
a decision procedure is extended admissible if and only if it is Bayes.

In \cref{sec: admissibility}, we employ the results of the previous section to connect admissibility and nonstandard Bayes optimality under various regularity conditions on the space and the nonstandard prior.  In the process, we give a nonstandard variant of Blyth's method.

In \cref{sec: examples}, we study several simple statistical decision problems to highlight the nonstandard theory
and its connections to the standard theory.

In \cref{sec: remarks}, we list some open problems.

\section{Standard preliminaries}
\label{sec: standardprelim}

A (nonsequential) statistical decision problem is defined in terms of
a \emph{parameter} space $\Theta$,  each element of which represents a possible state of nature;
a set $\AS$ of \defn{actions} available to the statistician;
a function $\Loss : \Theta \times \AS \to \NNReals$ characterizing the \defn{loss} associated with taking action $a \in \AS$ in state $\theta \in \Theta$;
and finally, a family $P=(P_\theta)_{\theta \in \Theta}$ of probability measures on a
measurable \emph{sample} space $X$.
On the basis of an observation from $P_{\theta}$ for some unknown element $\theta \in \Theta$,
the statistician decides to take a (potentially randomized) action $a$, and then suffers the loss $\Loss(\theta,a)$.

Formally, having fixed a $\sigma$-algebra on the space $\AS$ of actions,
every possible response by the statistician is
captured by a \defn{(randomized) decision procedure},
i.e., a map $\delta$ from $X$ to the space $\ProbMeasures{\AS}$ of probability measures $\AS$.
As is customary, we will write $\delta(x,A)$ for $(\delta(x))(A)$.
The expected loss, or \defn{risk}, to the statistician in state $\theta$ associated with following a decision procedure $\delta$
is
\[
\Risk_{\delta}(\theta)
= \Risk(\theta,\delta)
= \int_{X} \Bigl [ \int_{\AS} \Loss(\theta,a) \delta(x,\dee a) \Bigr ] \Model_{\theta}(\dee x).
\]
For the risk function to be well-defined, the maps
$x \mapsto \int_{\AS} \Loss(\theta,a) \delta(x,\dee a)$, for $\theta \in \Theta$, must be measurable,
and so we will restrict our attention to those decision procedures satisfying this weak measurability criterion.
A decision procedure $\delta$ is said to have finite risk if
$\Risk_{\delta}(\theta) \in \Reals$ for all $\theta \in \Theta$.
Let $\FRRE$ denote the set of randomized decision procedures with finite risk.

The set $\FRRE$ may be viewed as a convex subset of a vector space.
In particular, for all $\delta_1,\dotsc,\delta_n \in \FRRE$ and $p_1,\dotsc,p_n \in \NNReals$ with $\sum_i p_i = 1$,
define $\sum_i p_i \delta_i : X \to \ProbMeasures{\AS}$ by $(\sum_i p_i \delta_i)(x) = \sum_i p_i \delta_i(x)$ for $x \in X$.
Then $\Risk(\theta, \sum_i p_i \delta_i) = \sum_i p_i \, \Risk(\theta,\delta_i) < \infty$, and so we see that $\sum_i p_i \delta_i \in \FRRE$ and $r(\theta,\argdot)$ is a linear function on $\FRRE$ for every $\theta \in \Theta$.
For a subset $D \subseteq \FRRE$, let $\conv(D)$ denote the set of all finite convex combinations of decision procedures $\delta \in D$.

A decision procedure $\delta \in \FRRE$ is called \defn{nonrandomized}
if, for all $x \in X$, there exists $d(x) \in \AS$ such that $\delta(x,A) = 1$ if and only if $ d(x) \in A$, for all measurable sets $A \subseteq \AS$.
Let $\FiniteRiskEstimators \subseteq \FRRE$ denote the subset of all nonrandomized decision procedures.
Under mild measurability assumptions, %
every $\delta \in \FiniteRiskEstimators$ can be associated with a map $x \mapsto d(x)$ from $X$ to $\AS$
for which the risk satisfies
\[
\Risk(\theta,\delta) = \int_{X} \Loss(\theta,d(x)) \Model_{\theta}(\dee x).
\]
Finally, writing $\FiniteSubsets{S}$ for the set of all finite subsets of a set $S$,
let
\[
\FRREFS = \bigcup_{D \in \FiniteSubsets{\FiniteRiskEstimators}} \conv(D)
\]
be the set of randomized decision procedures that are finite convex combinations of nonrandomized decision procedures.
Note that $\FiniteRiskEstimators \subset \FRREFS \subset \FRRE$ and $\FRREFS$ is convex.

\subsection{Admissibility}
In general, the risk functions of two decision procedures are incomparable,
as one procedure may present greater risk in one state, yet less risk in another.
Some cases, however, are clear cut:
the notion of domination induces a partial order on the space of decision procedures.

\begin{definition}
Let $\epsilon \ge 0$ and $\delta,\delta' \in \FRRE$.
Then $\delta$ is \EpsDom{\epsilon} by $\delta'$ if
\begin{enumerate}
\item $\forall \theta\in \Theta$ $\Risk(\theta,\delta')\leq \Risk(\theta,\delta) - \epsilon$, and
\item $\exists \theta \in \Theta$ $\Risk(\theta,\delta') \neq \Risk(\theta,\delta)$.
\end{enumerate}
\end{definition}

Note that $\delta$ is \emph{dominated} by $\delta'$ if $\delta$ is \EpsDom{0} by $\delta'$.
If a decision procedure $\delta$ is $\epsilon$-dominated by another decision procedure $\delta'$, then, computational issues notwithstanding,
$\delta$ should be eliminated from consideration.  This gives rise to the following definition:

\begin{definition}
Let $\epsilon \ge 0$, $\cC \subseteq \FRRE$, and $\delta \in \FRRE$.
\begin{enumerate}
\item $\delta$ is \defn{\EpsAdmissible{\epsilon}{\cC}} unless $\delta$ is \EpsDom{\epsilon} by some $\delta' \in \cC$.
\item $\delta$ is \defn{\ExtAdmissible{\cC}} if $\delta$ is \EpsAdmissible{\epsilon}{\cC} for all $\epsilon > 0$.
\end{enumerate}
\end{definition}

Again, note that
$\delta$ is \defn{\Admissible{\cC}} if $\delta$ is \EpsAdmissible{0}{\cC}.
Clearly admissibility implies extended admissibility.
In other words, the class of all extended admissible decision procedures contains the class of all admissible decision procedures.

Admissibility leads to the notion of a complete class.

\begin{definition}
Let $\cA,\cC \subseteq \FRRE$.
Then $\cA$ is a \defn{complete} subclass of $\cC$ if,
for all $\delta \in \cC \setminus \cA$, there exists $\delta_{0} \in \cA$ such that $\delta_{0}$ dominates $\delta$.
Similarly, $\cA$ is an \defn{essentially complete} subclass of $\cC$ if,
for all $\delta \in \cC \setminus \cA$, there exists $\delta_{0} \in \cA$ such that $r(\theta,\delta_{0}) \le r(\theta,\delta)$ for all $\theta \in \Theta$.
An \defn{essentially complete class} is an essentially complete subclass of $\FRRE$.
\end{definition}

If a decision procedure $\delta$ is \Admissible{\cC}, then every complete subclass of $\cC$ must contain $\delta$.
Note that the term \defn{complete class} is usually used to refer to a complete subclass of some essentially complete class
(such as $\FRRE$ itself or $\FiniteRiskEstimators$ under the conditions described in \cref{convexity}.)

The next lemma captures a key consequence of essential completeness:
\begin{lemma}\label{esscomplete}
Suppose $\cA$ is an essentially complete subclass of $\cC$,
then \ExtAdmissible{\cA} implies \ExtAdmissible{\cC}.
\end{lemma}

The class of extended admissible estimators plays a central role in this paper.
It is not hard, however, to construct statistical decision problems for which the class is empty, and thus not a complete class.

\begin{example}
Consider a statistical decision problem with sample space $X=\{0\}$, parameter space $\Theta=\{0\}$, action space $\AS = (0,1]$,
and loss function $\Loss(0,d)=d$.
Then every decision procedure is a constant function, taking some value in $\AS$.
For all $c \in (0,1]$, the procedure  $\delta \equiv c$ is $c/2$-dominated by the decision procedure $\delta' \equiv c/2$.
Hence, there is no extended admissible estimator, hence the extended admissible procedures
do not form a complete class.
\end{example}

The following result gives conditions under which the class of extended admissible estimators are a complete class.
(See \citep[][\S5.4--5.6 and Thm.~5.6.3]{BG54} and \citep[][\S2.6 Cor.~1]{Ferguson} for related results for finite spaces.)

\begin{theorem}
Let $\cC \subseteq \FRRE$.
Suppose that, for all sequences $\delta, \delta_1, \delta_2, \dotsc \in \cC$ and
nondecreasing sequences $\epsilon_1,\epsilon_2,\dots \in\PosReals$ such that $\epsilon_{0} = \lim_i \epsilon_i$ exists
and $\delta$ is \EpsDom{\epsilon_{i}} by $\delta_{i}$ for all $i \in \Nats$,
there is a decision procedure $\delta_{0} \in \cC$ such that
$\delta$ is \EpsDom{\epsilon_{0}} by $\delta_{0}$.
Then the set of procedures that are \ExtAdmissible{\cC} form a complete subclass of $\cC$.
\end{theorem}
\begin{proof}
Let $\cS = \{ x \in \Reals^{\Theta} : (\exists \delta \in \cC)\,(\forall \theta \in \Theta)\, x(\theta) = r(\theta,\delta) \}$
denote the risk set of $\cC$. Pick $\delta \in \cC$ and suppose $\delta$ is not \ExtAdmissible{\cC}.
Let
\[
Q_{\epsilon}(\delta)=\{x\in \Reals^{\Theta}: (\forall \theta\in \Theta)(x(\theta)\leq \Risk(\theta,\delta)-\epsilon)\}.
\]
Let $M$ be the set $\{\epsilon\in \PosReals: Q_{\epsilon}(\delta)\cap \cS\neq \emptyset\}$,
which is nonempty because
$\delta$ is not \ExtAdmissible{\cC}.
As the risk is nonnegative and finite, $M$ is also bounded above.
Hence there exists a least upper bound $\epsilon_0$ of $M$.
Pick a non-decreasing sequence $\epsilon_1,\epsilon_2,\dotsc \in M$ that converges to $\epsilon_0$.
We now construct a (potentially infinite) sequence of decision procedures inductively:
\begin{enumerate}
\item Choose $\delta_1 \in \cC$ such that $\delta$ is \EpsDom{\epsilon_{1}} by $\delta_1$. Because $M$ is nonempty, there must exist such a procedure.

\item Suppose we have chosen $\delta_1,\dots,\delta_i \in \cC$, and suppose there is an index $j \in \Nats$ such that
$\delta$ is \EpsDom{\epsilon_{j}} by $\delta_{i}$ but
$\delta$ is not \EpsDom{\epsilon_{j+1}} by $\delta_{i}$.
Then we choose $\delta_{i+1} \in \cC$ such that
$\delta$ is \EpsDom{\epsilon_{j+1}} by $\delta_{i+1}$.  Because $M$ contains $\epsilon_{j+1}$, there must exist such a procedure.
If no such index $j$ exist, the process halts at stage $i$.

\end{enumerate}

Suppose the process halts at some finite stage $i_0$.
Then for, all $j \in \Nats$,
$\delta$ is not \EpsDom{\epsilon_{j}} by $\delta_{i_{0}}$ or
$\delta$ is \EpsDom{\epsilon_{{j+1}}} by $\delta_{i_{0}}$.
But $\delta$ is \EpsDom{\epsilon_{1}} by $\delta_{i_{0}}$ and
so,
by induction,
$\delta$ is \EpsDom{\epsilon_{j}} by $\delta_{i_{0}}$ for all $j \in \Nats$.
As the sequence $\epsilon_1,\epsilon_2,\dots$ is non-decreasing and has a limit $\epsilon_{0}$,
it follows easily via a contrapositive argument that
$\delta$ is even \EpsDom{\epsilon_{0}} by $\delta_{i_{0}}$.
If $\delta_{i_{0}}$ were not \ExtAdmissible{\cC}, then this would contradict the fact that $\epsilon_{0}$ is a least upper bound on $M$.

Now suppose the process continues indefinitely.
Then the claim is that $\delta$ is \EpsDom{\epsilon_{i}} by $\delta_{i}$ for all $i \in \Nats$.
Clearly this holds for $i=1$. Supposing it holds for $i \le k$.
Then $\delta$ is \EpsDom{\epsilon_{i}} by $\delta_{k}$ for all $i \le k$
and there exists $j \in \Nats$
such that
$\delta$ is \EpsDom{\epsilon_{j}} by $\delta_{k}$ but
$\delta$ is not \EpsDom{\epsilon_{j+1}} by $\delta_{k}$.
It follows that $j \ge k$, hence
$\delta$ is \EpsDom{\epsilon_{k+1}} by $\delta_{k+1}$, as was to be shown.

Thus, by hypothesis,
there is a decision procedure $\delta' \in \cC$ such that
$\delta$ is \EpsDom{\epsilon_{0}} by $\delta'$.
As $\epsilon_0$ is the least upper bound of $M$, $\delta'$ is also \ExtAdmissible{\cC},
completing the proof.
\end{proof}

\subsection{Bayes optimality}
Consider now the Bayesian framework, in which one adopts a \defn{prior}, i.e., a probability measure $\pi$ defined on some $\sigma$-algebra on $\Theta$.
Irrespective of the interpretation of $\pi$,
we may define the \defn{Bayes} risk of a procedure
as the expected risk under a parameter chosen at random from $\pi$.\footnote{
We must now also assume that $r(\argdot,\delta)$ is a measurable function for every $\delta \in \FRRE$.
Normally, there is a natural choice of $\sigma$-algebra on $\Theta$ that satisfies this constraint.
Even if there is no natural choice, there is always a sufficiently rich $\sigma$-algebra that renders every risk function measurable.
In particular, the power set of $\Theta$ suffices.  Note that the $\sigma$-algebra determines the set of possible prior distributions.
In the extreme case where the $\sigma$-algebra on $\Theta$ is taken to be the entire power set,
the set of prior distributions contain the purely atomic distributions and these are the
only distributions 
if and only if 
there is no real-valued measurable cardinal less than or equal to the continuum \citep[Thm.~1D]{Fremlin}.
As we will see, the purely atomic distributions suffice to give our complete class theorems.
}

\begin{definition}\label{bayesdef}
Let $\delta \in \FRRE$, $\epsilon \ge 0$, and $\cC \subseteq \FRRE$, and let $\pi_0$ be a prior.
\begin{enumerate}
\item
The \defn{Bayes risk under $\pi_0$ of $\delta$} is
$\Risk(\pi_0,\delta)=\int_{\Theta}\Risk(\theta,\delta)\pi_0(\dee \theta)$.

\item
$\delta$ is \defn{\eBayesP{\epsilon}{\pi_0}{\cC}} if
$\Risk(\pi_0,\delta) < \infty$ and,
for all $\delta'\in \cC$, we have $\Risk(\pi_0,\delta) \leq \Risk(\pi_0,\delta') + \epsilon$.

\item
$\delta$ is \defn{\BayesP{\pi_0}{\cC}} if
$\delta$ is \eBayesP{0}{\pi_0}{\cC}.

\item
$\delta$ is \defn{extended \Bayes{\cC}} if, for all $\epsilon > 0$, there exists a prior $\pi$ such that
$\delta$ is \eBayesP{\epsilon}{\pi}{\cC}.

\item
$\delta$ is \defn{\eBayes{\epsilon}{\cC}} (resp., \defn{\Bayes{\cC}})
if there exists a prior $\pi$ such that $\delta$ is \eBayesP{\epsilon}{\pi}{\cC} (resp., \BayesP{\pi}{\cC}).

\end{enumerate}
\end{definition}

We will sometimes write \defn{\Bayes{\cC} with respect to $\pi_0$} to mean \BayesP{\pi_0}{\cC},
and similarly for \eBayes{\epsilon}{\cC}.

The following well-known result establishes a basic connection between Bayes optimality and admissibility (see, e.g., \citep[][Thm.~5.5.1]{BG54}).
We give a proof for completeness.
\begin{theorem}\label{BayesImpliesExtAdm}
If $\delta$ is \Bayes{\cC}, then
$\delta$ is extended \Bayes{\cC}, and then
$\delta$ is \ExtAdmissible{\cC}.
\end{theorem}
\begin{proof}
That Bayes implies extended Bayes follows trivially from definitions.
Now assume $\delta$ is not \ExtAdmissible{\cC}. Then there exists $\epsilon > 0$ and $\delta' \in \cC$ such that
$\Risk(\theta,\delta') \le \Risk(\theta,\delta) - \epsilon$ for all $\theta \in \Theta$.
But then, for every prior $\pi$,
$\int \Risk(\theta,\delta')\pi(\dee \theta) \le \int \Risk(\theta,\delta)\pi(\dee \theta) - \epsilon$
or
$\int \Risk(\theta,\delta')\pi(\dee \theta) = \int \Risk(\theta,\delta)\pi(\dee \theta) = \infty$,
hence $\delta$ is not \eBayes{\epsilon/2}{\cC}, hence not extended \Bayes{\cC}.
\end{proof}

Note that neither extended admissibility nor admissibility imply Bayes optimality, in general.
E.g., the maximum likelihood estimator in a univariate normal-location problem is admissible, but not Bayes.

Essential completeness allows us to strengthen a Bayes optimality claim:

\begin{theorem}\label{EssCompleteBayes}
Suppose $\cA$ is an essentially complete subclass of $\cC$,
then \eBayes{\epsilon}{\cA} implies \eBayes{\epsilon}{\cC} for every $\epsilon \ge 0$.
\end{theorem}
\begin{proof}
Let $\delta_{0}$ be \BayesP{\pi}{\cA} for some prior $\pi$.
Let $\delta \in \cC$.  Then there exists $\delta' \in \cA$ such that,
for all $\Risk(\theta,\delta') \le \Risk(\theta,\delta)$ for all $\theta \in \Theta$.
By hypothesis, $\Risk(\pi,\delta_{0}) \le \Risk(\pi,\delta')$, but
$\Risk(\pi,\delta') = \int \Risk(\theta,\delta') \pi(\dee\theta)
\le \int \Risk(\theta,\delta) \pi(\dee\theta)  = \Risk(\pi,\delta)$.
Hence $\Risk(\pi,\delta_{0}) \le \Risk(\pi,\delta)$ for all $\delta \in \cC$.
\end{proof}

\subsection{Convexity}
\label{convexity}

An important class of statistical decision problems are those in which the action space $\AS$ is itself a vector space over the field $\Reals$.
In that case, the mean estimate
 $\int_{\AS} a\, \delta(x,\dee a)$ is well defined for every $\delta \in \FRREFS$ and $x \in X$,
 which motivates the following definition.

\begin{definition}\label{averageEst}
For $\delta \in \FRREFS$,
define $\EEst{\delta} : X \to \ProbMeasures{\AS}$ by
$\EEst{\delta}(x,A) = 1$ if $ \int_{\AS} a\, \delta(x,\dee a) \in A $ and $0$ otherwise,
for every $x \in X$ and measurable subset $A \subseteq \AS$.
\end{definition}

When the loss function is assumed to be convex, it is well known that the mean action will be no worse on average than the original randomized one.
We formalize this condition below and prove several well-known results for completeness.

\begin{condition}{LC}[loss convexity]
\label{assumptioncv}
$\AS$ is a vector space over the field $\Reals$ and the loss function $\Loss$ is convex with respect to the second argument.
\end{condition}

\begin{lemma}\label{convexess}
Let $\delta$ and $\EEst{\delta}$ be as in \cref{averageEst},
and suppose \cref{assumptioncv} holds.
Then $\Risk(\argdot, \delta) \ge \Risk(\argdot, \EEst{\delta})$, hence $\EEst{\delta} \in \FiniteRiskEstimators$.
\end{lemma}

\begin{proof}
Let $\theta \in \Theta$.
By convexity of $\Loss$ in its second parameter and a finite-dimensional version of Jensen's inequality~\citep[][\S2.8 Lem.~1]{Ferguson},
we have
\[
\Risk(\theta, \delta)
&= \int_{X} \Bigl [ \int_{\AS} \Loss(\theta,a) \delta(x,\dee a) \Bigr ] \Model_{\theta}(\dee x)
\\&\ge \int_{X}  \Loss(\theta, {\textstyle \int_{\AS} a \,\delta(x,\dee a) } )  \Model_{\theta}(\dee x)
= \Risk(\theta, \EEst{\delta}).
\]
\end{proof}

\begin{remark}
Irrespective of the dimensionality of the action space $\AS$, we may use a finite-dimensional version of Jensen's inequality because the procedure $\delta \in \FRREFS$ is a finite mixture of nonrandomized procedures.  The proof for a general randomized procedure $\delta \in \FRRE$ and a general action space $\AS$, would require additional hypotheses  to account for the possible failure of Jensen's inequality (see \citep[][]{MR0362421}) and the
possible lack of measurability of $\EEst{\delta}$ (see \citep[][S2.8]{Ferguson}).
\end{remark}

\begin{lemma}\label{FREconvexECC}
Suppose \cref{assumptioncv} holds.
Then $\FiniteRiskEstimators$ is an essentially complete subclass of $\FRREFS$.
\end{lemma}
\begin{proof}
Let $\delta\in \FRREFS$.
Then $\EEst{\delta} \in \FiniteRiskEstimators$.
By \cref{convexess}, $\EEst{\delta}$ is well defined and $\Risk(\theta,\delta_0) \ge \Risk(\theta,\EEst{\delta})$, completing the proof.
\end{proof}

\begin{remark}
See the remark following \citep[][\S2.8 Thm.~1]{Ferguson} for a discussion of additional hypotheses needed for establishing that $\FiniteRiskEstimators$ is an essentially complete subclass of $\FRRE$.
\end{remark}

\section{Prior work}
\label{sec:prior}

The first key results on admissibility and Bayes optimality are due to Abraham Wald,
who laid the foundation of sequential decision theory.
In \citep{Wald47},
working in the setting of sequential statistical decision problems
with compact parameter spaces,
Wald showed that
the Bayes decision procedures form an essentially complete class.
Sequential decision problems differ from the decision problems we will be discussing in this paper in the sense that it gives the statistician the freedom to look at a sequence of observations one at a time and to decide, after each observation, whether to stop and take an action or to continue, potentially at some cost.
The decision problems we will be discussing in this paper can be seen as special cases of sequential decision problems with only one observation.

In order to prove his results, Wald required a strong form of continuity for his risk and loss functions.
\begin{definition}\label{pintrin}
A sequence of parameters $\{\theta_i\}_{i\in \Nats}$ converges \defn{in risk} to a parameter $\theta$
when
$\sup_{\delta \in \FRRE}  | \Risk(\theta_i,\delta)  - \Risk(\theta,\delta)|  \to 0$  as $i \to \infty$,
and converges
\defn{in loss} when
$\sup_{a \in \AS} |\Loss(\theta_i,a) - \Loss(\theta,a) | \to 0$ as $ i \to \infty$.
Similarly, a sequence of decision procedures $\{\delta_i\}_{i\in \Nats}$ in $\FRRE$ converges \defn{in risk} to a decision procedure $\delta$
when
$\sup_{\theta \in \Theta} | \Risk(\theta,\delta_i) - \Risk(\theta,\delta) | \to 0$ as $i \to \infty$.
A sequence of actions $\{a_i\}_{i\in \Nats}$
converges
\defn{in loss} to an action $a \in \AS$ when
$\sup_{\theta \in \Theta} |\Loss(\theta,a_i) - \Loss(\theta,a) | \to 0$ as $i \to \infty$.
\end{definition}

Topologies on $\Theta$, $\AS$, and $\FRRE$ are generated by these notions of convergence.
In the following result and elsewhere, a model $P$ is said to admit (a measurable family of) densities $( f_{\theta})_{\theta \in \Theta}$ (with respect to a dominating
($\sigma$-finite) measure $\nu$) when
$P_\theta(A) = \int_A f_\theta(x)\, \nu(\dee x)$ for every $\theta \in \Theta$ and measurable $A \subseteq X$.
In terms of these densities,
there is a unique Bayes solution with respect to a prior $\pi$ on $\Theta$
when,
for every $x \in X$, except perhaps for a set of $\nu$-measure $0$,
there exists one and only one action $a^* \in \AS$ for which the expression
\[
\int_{\Theta} \Loss(\theta,a)f_\theta(x)\, \pi(\dee \theta)
\]
takes its minimum value with respect to $a \in \AS$. (Another notion of uniqueness used in the literature is to simply demand that the risk functions of two Bayes solutions agree.)
The main result can be stated in the special case of a nonsequential decision problem as follows:
\begin{theorem}[{\citep[][Thms.~4.11 and~4.14]{Wald47}}]
Assume $\Theta$ and $\FRRE$ are compact in risk, and that $\Theta$ and $\AS$ are compact in loss.
Assume further that $P$ admits densities $( f_{\theta})_{\theta \in \Theta}$ with respect to Lebesgue measure,
that these densities are strictly positive outside a Lebesgue measure zero set.
Then every extended admissible decision procedure is Bayes.
If the Bayes solution for every prior $\pi$ is unique,
the class of nonrandomized Bayes procedures form a complete class.
\end{theorem}

Wald's regularity conditions
are quite strong; he essentially requires equicontinuity in each variable for both the loss and risk functions.
For example, the standard normal-location problem under squared error does not satisfy these criteria.

A similar result is established in the nonsequential setting in \citep{Wald47a}:
\begin{theorem}[{\citep[][Thm.~3.1]{Wald47a}}]
\label{waldcompact}
Suppose
that $P$ admits densities $( f_{\theta})_{\theta \in \Theta}$,
that $\Theta$ is a compact subset of a Euclidean space,
that the map $(x,\theta) \mapsto f_\theta(x)$ is jointly continuous,
that the loss $\Loss(\theta,a)$ is a continuous function of $\theta$ for every action $a$,
that the space $\AS$ is compact in loss,
 and that there is a unique Bayes solution for every prior $\pi$ on $\Theta$.
Then every Bayes procedure is admissible and
 the collection of Bayes procedures form an essentially complete class.
 \end{theorem}

In many classical statistical decision problems, one does not lose anything by assuming that
all risk functions are continuous. The following theorem, taken from \citep{LC98}, formalizes this intuition:
We will say that a model $P$ has a continuous likelihood function $( f_{\theta})_{\theta \in \Theta}$
when $P$
admits densities $( f_{\theta})_{\theta \in \Theta}$
such that $\theta \mapsto f_{\theta}(x)$ is continuous for every $x \in X$.

\begin{theorem}[{\citep[][\S5 Thm.~7.11]{LC98}}]\label{bdrisk}
Suppose $P$  has a continuous likelihood function $( f_{\theta})_{\theta \in \Theta}$
and a monotone likelihood ratio.
If the loss function $\Loss(\theta,\delta)$ satisfies
\begin{enumerate}
\item $\Loss(\theta,a)$ is continuous in $\theta$ for each action $a$;

\item $\Loss(\theta,a)$ is decreasing in $a$ for $a<\theta$ and increasing in $a$ for $a>\theta$; and

\item there exist functions $f$ and $g$, which are bounded on all bounded subsets of $\Theta \times \Theta$, such that for all $a$
\[
\Loss(\theta,a)\leq f(\theta,\theta')\Loss(\theta',a)+g(\theta,\theta'),
\]
\end{enumerate}
then the estimators with finite-valued, continuous risk functions form a complete class.
\end{theorem}

If we assume the loss function is bounded, then all decision procedures have finite risk. The following theorem gives a characterization of continuous risk assuming boundedness of the loss.

\begin{theorem}[{\citep[][\S3.7 Thm.~1]{Ferguson}}]\label{bdlossbdr}
Suppose $P$ admits densities $(f_{\theta})_{\theta\in \Theta}$ with respect to a dominating measure $\nu$.
Assume
\begin{enumerate}
\item $\Loss$ is bounded;
\item $\Loss(\theta,a)$ is continuous in $\theta$, uniformly in $a$;
\item for every bounded measurable $\phi$, $\int \phi(x)f_{\theta}(x)\nu(\dee x)$ is continuous in $\theta$.
\end{enumerate}
Then the risk $\Risk(\theta,\delta)$ is continuous in $\theta$ for every $\delta$.
\end{theorem}

If we assume continuity of the risk function with respect to the parameter and restrict ourselves to Euclidean parameter spaces, we have the following theorem from \citep[][Sec.~8.8, Thm.~12]{Berger85}.

\begin{theorem}\label{bergercompact}
Assume that $\AS$ and $\Theta$ are compact subsets of Euclidean spaces and
that the model $P$ admits densities $(f_{\theta})_{\theta\in \Theta}$ with respect to either Lebesgue or counting measure
such that the map $(x, \theta) \mapsto f_{\theta}(x)$ is jointly continuous.
Assume further that the loss $\Loss(\theta,a)$ is a continuous function of $a \in \AS$ for each $\theta$,
and that all decision procedures have continuous risk functions.
Then the collection of Bayes procedures form a complete class.
\end{theorem}

In the noncompact setting, Bayes procedures generally do not form a complete class.
With a view to generalizing the notion of a Bayes procedure and recovering a complete class,
Wald \citep{Wald49} introduced the notion of
``Bayes in the wide sense", which we now call extended Bayes (see \cref{bayesdef}).
The formal statement of the following theorem is adapted from \citep{Ferguson}:

\begin{theorem}
Suppose that there exists a topology on $\FRRE$ such that
$\FRRE$ is compact and
$\Risk(\theta,\delta)$ is lower semicontinuous in $\delta\in \FRRE$ for all $\theta\in \Theta$.
Then the set of extended Bayes procedures form an essentially complete class.
\end{theorem}

Wald also studied taking the ``closure" (in a suitable sense) of the collection of all Bayes procedures,
and showed that every admissible procedure was contained in this new class.
The first result of this form appears in \citep{Wald49} and is extended later in \citep{LeCam55}.
\citet[][App.~4A]{Brown86} extended these results and gave a modern treatment.
The following statement of Brown's version is adapted from \citep[][\S5 Thm.~7.15]{LC98}.

\begin{theorem}
Assume $P$ admits strictly positive densities $(f_{\theta})_{\theta\in \Theta}$ with respect to a $\sigma$-finite measure $\nu$.
Assume the action space $\AS$ is a closed convex subset of Euclidean space.
Assume the loss $\Loss(\theta,a)$ is lower semicontinuous and strictly convex in $a$ for every $\theta$, and satisfies
\[
\lim_{|a|\to \infty}\Loss(\theta,a)=\infty \text{ for all $\theta\in \Theta$}.
\]
Then every admissible decision procedure $\delta$ is an a.e.\ limit of Bayes procedures, i.e.,
there exists a sequence $\pi_n$ of priors with support on a finite set,
such that
\[
\delta^{\pi_n}(x)\to \delta(x) \text{ as $n \to \infty$ for $\nu$-almost all $x$},
\]
where $\delta^{\pi_n}$ is a Bayes procedure with respect to $\pi_n$.
\end{theorem}

In the normal-location model under squared error loss, the sample mean, while not a Bayes estimator in the strict sense,
can be seen as a limit of Bayes estimators, e.g., with respect to normal priors of variance $K$ as $K \to \infty$ or uniform priors on $[-K,K]$ as $K \to \infty$. (We revisit this problem in \cref{normallocprob}.)
In his seminal paper, Sacks~\citep{Sacks63} observes that the sample mean is also the Bayes solution if the notion of prior distribution is relaxed to include Lebesgue measure on the real line.
Sacks~\citep{Sacks63} raised the natural question: if $\delta$ is a limit of Bayes estimators, is there a measure $m$ on the real line such that $\delta$ is ``Bayes" with respect to this measure? A solution in this latter form was termed a \emph{generalized Bayes solution} by \citet{Sacks63}.
The following definition is adapted from \citep{stone67}:

\begin{definition}
A decision procedure $\delta_0$ is a \defn{normal-form generalized Bayes procedure} with respect to a $\sigma$-finite measure $\pi$ on $\Theta$
when $\delta_m$ minimizes
$\Risk(\pi, \delta) = \int \Risk(\theta,\delta)\pi(\dee \theta)$, subject to the restriction that $\Risk(\pi,\delta_m) < \infty$.
If $P$ admits densities $(f_{\theta})_{\theta\in \Theta}$ with respect to a $\sigma$-finite measure $\nu$
and $\delta_0$ minimizes the unnormalized posterior risk $\int \Loss(\theta, \delta_0(x))\,f_\theta(x)\, \pi(\dee \theta)$ for $\nu$-a.e.\ $x$,
then $\delta_0$ is a \defn{(extensive-form) generalized Bayes procedure} with respect to $\pi$.
\end{definition}

When a model admits densities, \citet{stone67} showed that every normal-form generalized Bayes procedure is also extensive-form. (\citeauthor{Sacks63} defined generalized Bayes in extensive form, but demanded also that $\int f_\theta(\argdot)\, \pi(\dee \theta)$ be finite $\nu$-a.e.
The notion of normal- and extensive-form definitions of Bayes optimality were introduced by \citet{RaiffaSchlaifer61}.)
For exponential families, under suitable conditions, one can show that every admissible estimator is generalized Bayes.
The first such result was developed by \citet{Sacks63} in his original paper: he proved that,
for statistical decision problems where the model admits a density of the form $e^{x\theta} / Z_\theta$ with $Z_\theta =\int e^{x\theta}\nu(\dee \theta)$, every admissible estimator is generalized Bayes.
\citet{stone67} extended this result to estimation of the mean
in one-dimensional exponential families under squared error loss.
These results were further generalized in similar ways by \citet[Sec.~3.1]{Brown71} and \citet{Berger78}.
The following theorem is given in \citep{Berger78}. We adapt the statement of this theorem from \citep{LC98}.

\begin{theorem}[{\citep[\S5 Thm.~7.17]{LC98}}] \label{berger}
Assume the model is a finite-dimensional exponential family, 
and that the loss $\Loss(\theta,a)$ is jointly continuous, strictly convex in $a$ for every $\theta$, and satisfies
\[
\lim_{|a| \to \infty} \Loss(\theta,a)=\infty\  \text{ for all $\theta\in \Theta$}.
\]
Then every admissible estimator is generalized Bayes.
\end{theorem}

Other generalized notions of Bayes procedures have been proposed.
\citet{Heath78} study statistical decision problems in the setting of finitely additive probability spaces.
The following theorem is their main result:

\begin{theorem}[{\citep[][Thm.~2]{Heath78}}]
Fix a class $\FRRE$ of decision procedures.
Every finitely additive Bayes decision procedure
is extended admissible.
If the loss function is bounded and the class $\FRRE$ is convex,
then every extended admissible decision procedure in $\FRRE$ is finitely additive Bayes in $\FRRE$.
\end{theorem}

The simplicity of this statement is remarkable. However,
the assumption of boundedness is very strong, and rule out many standard estimation problems on unbounded spaces.
We will succeed in removing the boundedness assumption by moving to a sufficiently saturated nonstandard model.

\section{Nonstandard admissibility}
\label{sec: nonstandardadmis}

As we have seen in the previous section,
strong regularity appears to be necessary to align Bayes optimality and admissibility.
In noncompact parameter spaces, the statistician must apparently abandon the strict use of probability measures in order to represent certain extreme states of uncertainty that correspond with admissible procedures. 
Even then, strong regularity conditions are required (such as domination of the model and strict positiveness of densities, ruling out estimation in infinite-dimensional contexts).
In the remainder of the paper, we describe a new approach using nonstandard analysis, in which the statistician
retains the use of probability measures, but has access to a much richer collection of real numbers to express their beliefs.

Let
$(\Theta,\AS,\Loss,X,\Model)$
be a standard statistical decision problem.

We will assume the reader is familiar with basic concepts and key results in nonstandard analysis.
(See \cref{intronsa,sec:inp} for a review tailored to this paper.)
For a set $S$, let $\PowerSet{S}$ denote its power set.
We assume that we are working within a nonstandard model containing
$V \supseteq \Reals \cup \Theta \cup \AS \cup X$,
$\PowerSet{V}, \PowerSet{V \cup \PowerSet{V}}, \dotsc$, and we assume the model is as saturated as necessary.
We use $\NSE{}$ to denote the nonstandard extension map taking elements, sets, functions, relations, etc., to their nonstandard counterparts.
In particular, $\HReals$ and $\NSE{\Nats}$ denote the nonstandard extensions of the reals and natural numbers, respectively.
Given a topological space $(Y,T)$ and a subset $X \subseteq \NSE{Y}$,
let $\NS{X} \subseteq X$  denote the subset of near-standard elements (defined by the monadic structure induced by $T$)
and let $\ST : \NS{Y} \to Y$ denote the standard part map taking near-standard elements to their standard parts.
In both cases, the notation elides the underlying space $Y$ and the topology $T$, because the space and topology will always be clear from context.
As an abbreviation,
we will write $^\circ x$ for $\ST(x)$
for atomic elements $x$.
For functions $f$, we will write $^\circ f$ for the composition $x \mapsto \ST(f(x))$.
Finally, given an internal (hyperfinitely additive) probability space $(\Omega,\cF,P)$,
we will write $(\Omega,\Loeb{\cF},\Loeb{P})$ to denote the corresponding Loeb space, i.e., the completion of the unique extension of $P$ to $\sigma(\cF)$.

\subsection{Nonstandard extension of a statistical decision problem}

We will assume that $\Theta$ is a Hausdorff space and adopt its Borel $\sigma$-algebra $\BorelSets{\Theta}$.\footnote{
In one sense,  this is a mild assumption, which we use to ensure that the standard part map $\ST{} : \NS{\NSE{\Theta}} \to \Theta$ is well-defined.
In another sense,
$\Theta$ can always be made Hausdorff by, e.g., adopting the discrete topology.
The topology determines the Borel sets and thus determines the set of available probability measures on $\Theta$ (and on $\NSE{\Theta}$, by extension).
Topological considerations arise again in \cref{sec: CCs under compactness and RC},
\cref{remtopoad}, and
\cref{remarktopo}.
}

One should view the model $\Model$ as a function from $\Theta$ to the space $\ProbMeasures{X}$ of probability measures on $X$.  Write $\NSE{\Model}_y$ for $(\NSE{\Model})_y$.
For every $y \in \NSE{\Theta}$, the transfer principle implies that $\NSE{\Model}_y$ is an internal probability measure on $\NSE{X}$ (defined on the extension of its $\sigma$-algebra).
By \cref{nsfamilyresult}, we know that $\NSE{(\Model_\theta)}  = \NSE{\Model}_{\theta}$ for $\theta\in \Theta$, as one would expect from the notation.
Recall that standard decision procedures $\delta \in \FRRE$ have finite risk functions.
Therefore, the risk map $(\theta,\delta)\mapsto \Risk(\theta,\delta)$ is a function from $\Theta \times \FRRE$ to $\Reals$.
By the extension and transfer principles,
the nonstandard extension $\sRisk$ is an internal function from $\NSE{\Theta}\times \sFRRE$ to $\HReals$.
and $\sdelta \in \sFRRE$ if $\delta \in \FRRE$.
The transfer principle also implies that every $\Delta\in \sFRRE$ is an internal function
from $\NSE{X}$ to $\NSE{\!\!\ProbMeasures{\AS}}$.
The $\NSE{}$risk function of $\Delta \in \sFRRE$ is the function $\sRisk(\argdot, \Delta)$ from $\NSE{\Theta}$ to $\HReals$.
By the transfer of the equation defining risk, the following statement holds:
\[
(\forall \theta\in \NSE{\Theta})\ (\forall \Delta \in \sFRRE)\
(\sRisk(\theta,\Delta)
= \sint{\NSE{X}} \Bigl [ \sint{\NSE{\AS}} \NSE{\Loss}(\theta,a) \Delta(x,\dee a) \Bigr ] \NSE{\Model}_{\theta}(\dee x).
\]
As is customary, we will simply write $\int$ for $\NSE{\!\!\int}$, provided the context is clear. (We will also drop $\NSE{}$ from the extensions of common functions and relations like addition, multiplication, less-than-or-equal-to, etc.)

\subsection{Nonstandard admissibility}

Let $\delta_{0},\delta \in \FRRE$,
let $\epsilon \in \NNReals$,
and assume
$\delta_{0}$ is \EpsDom{\epsilon} by $\delta$.
Then there exists $\theta_{0} \in \Theta$
such that
\[
(\forall \theta\in \Theta)(\Risk(\theta,\delta)\leq \Risk(\theta,\delta_0)-\epsilon)\land (\Risk(\theta_0,\delta)\neq \Risk(\theta_0,\delta_0)).
\]
By the transfer principle,
\[
(\forall \theta\in \NSE{\Theta})(\sRisk(\theta,\sdelta)\leq \sRisk(\theta,\sdelta_{0})-\epsilon) \land
(\sRisk(\theta_{0},\sdelta)\neq \sRisk(\theta_{0},\sdelta_0)).
\]
Because $\sRisk(\theta_{0},\sdelta) = \Risk(\theta_{0},\delta)$ and similarly for $\sRisk(\theta_{0},\sdelta_{0})$,
\cref{standardnums}.1 implies that $\sRisk(\theta_{0},\sdelta) \not\approx \sRisk(\theta_{0},\sdelta_0)$.
These results motivate the following nonstandard version of domination.

\begin{definition}\label{epsdef}
Let $\Delta,\Delta' \in \sFRRE$ be internal decision procedures,
let $\epsilon \in \NNReals$,
and $R , S \subseteq \NSE{\Theta}$.
Then $\Delta$ is \defn{\spEpsDomIn{\epsilon}{R}{S}} by $\Delta'$ when 

\begin{enumerate}
\item $\forall \theta \in S$ \ $\sRisk(\theta,\Delta') \le \sRisk(\theta,\Delta) - \epsilon$, and
\item $\exists \theta \in R$ \ $\sRisk(\theta,\Delta') \not\approx \sRisk(\theta,\Delta)$.
\end{enumerate}

Write
\defn{\spDomIn{R}{S}}
for
\spEpsDomIn{0}{R}{S},
and write \defn{\spEpsDom{\epsilon}{S}}
for
\defn{\spEpsDomIn{\epsilon}{S}{S}}.
\end{definition}

The following results are immediate upon inspection of the definition above, and the fact
that (1) implies (2) for $R \subseteq S$ when $\epsilon > 0$.
\begin{lemma}\label{domlattice}
Let $\epsilon \le \epsilon'$,
$R \subseteq R'$, and
$S \subseteq S'$.
Then
\spEpsDomIn{\epsilon'}{R}{S'}
implies
\spEpsDomIn{\epsilon}{R'}{S}.
If $\epsilon > 0$,
then
\spEpsDomIn{\epsilon}{S}{S'}
if and only if
\spEpsDom{\epsilon}{S'},
and
\spEpsDom{\epsilon'}{S'}
implies
\spEpsDom{\epsilon}{S}.
\end{lemma}

The following result connects standard and nonstandard domination.

\begin{theorem}\label{domthm}
Let $\epsilon \in \NNReals$ and $\delta_{0},\delta \in \FRRE$.
The following statements are equivalent:
\begin{enumerate}
\item $\delta_{0}$ is \EpsDom{\epsilon} by $\delta$.
\item $\sdelta_{0}$ is \spEpsDomIn{\epsilon}{\Theta}{\NSE{\Theta}} by $\sdelta$.
\item $\sdelta_{0}$ is \spEpsDom{\epsilon}{\Theta} by $\sdelta$.
\end{enumerate}
If $\epsilon > 0$, then the following statement is also equivalent:
\begin{enumerate}[start=4]
\item $\sdelta_{0}$ is \spEpsDom{\epsilon}{\NSE{\Theta}} by $\sdelta$.
\end{enumerate}
\end{theorem}
\begin{proof}
$(1 \implies 2$) Follows from logic above \cref{epsdef}.
($2 \implies 3$) Follows from \cref{domlattice}.
($3 \implies 1$) By hypothesis,
\[
(\forall \theta\in \Theta)(\sRisk(\theta,\sdelta) \le \sRisk(\theta,\sdelta_{0})-\epsilon)
\land (\exists \theta_{0} \in \Theta)(\sRisk(\theta_{0},\sdelta) \not\approx \sRisk(\theta_{0},\sdelta_0)).
\]
Because
$\sRisk(\theta_{0},\sdelta) = \Risk(\theta_{0},\delta)$, and likewise for $\delta_{0}$,
it follows that
\[
(\forall \theta\in \Theta)(\Risk(\theta,\delta) \le \Risk(\theta,\delta_{0})-\epsilon).
\]
Similarly, $\SP{(\sRisk(\theta_{0},\sdelta))} = \Risk(\theta_{0},\delta)$, and likewise for $\delta_{0}$,
hence \cref{standardnums}.1 implies
\[
 (\exists \theta_{0} \in \Theta) (\Risk(\theta_{0},\delta) \neq \Risk(\theta_{0},\delta_0)).
\]
($2 \implies 4 \implies 3$) Follow from \cref{domlattice}.
\end{proof}

\begin{definition}
Let $\epsilon \in \NNReals$, $R,S \subseteq \NSE{\Theta}$, and $\cC \subseteq \sFRRE$,
and $\Delta \in \sFRRE$.
\begin{enumerate}

\item
$\Delta$ is \defn{\spEpsAdmissibleIn{\epsilon}{R}{S}{\cC}}
unless $\Delta$ is \spEpsDomIn{\epsilon}{R}{S} by some $\Delta' \in \cC$.

\item
$\Delta$ is \defn{\spAdmissibleIn{R}{S}{\cC}}
if $\Delta$ is \spEpsAdmissibleIn{0}{R}{S}{\cC}.

\item
$\Delta$ is \defn{\spEpsAdmissible{\epsilon}{S}{\cC}}
if $\Delta$ is \spEpsAdmissibleIn{\epsilon}{S}{S}{\cC}.

\item
$\Delta$ is \defn{\spExtAdmissible{S}{\cC}}
if $\Delta$ is \spEpsAdmissible{\epsilon}{S}{\cC} for every $\epsilon \in \PosReals$.

\end{enumerate}
\end{definition}

The following result is immediate upon inspection of the definitions above.
\begin{lemma}\label{lattice}
Let $\epsilon \le \epsilon'$,
$R \subseteq R'$,
$S \subseteq S'$, and
$\cA \subseteq \cC$.
Then
\spEpsAdmissibleIn{\epsilon}{R'}{S}{\cC}
implies
\spEpsAdmissibleIn{\epsilon'}{R}{S'}{\cA}.
For $\epsilon > 0$,
\spEpsAdmissible{\epsilon}{S}{\cC}
implies
\spEpsAdmissible{\epsilon'}{S'}{\cA}.
\end{lemma}

The analogous results for \spAdmissibleIn{R}{S}{\cC} and \spExtAdmissible{S}{\cC} then follow immediately.
The following result connects standard and nonstandard admissibility.

\begin{theorem}\label{admissibilitythm}
Let $\epsilon \in \NNReals$,
$\delta_{0} \in \FRRE$,
and $\cC \subseteq \FRRE$.
Define $\Ext{\cC} = \{\sdelta : \delta \in \cC\}$.
Then the following statements are equivalent:
\begin{enumerate}
\item $\delta_{0}$ is \EpsAdmissible{\epsilon}{\cC}.
\item $\sdelta_{0}$ is \spEpsAdmissibleIn{\epsilon}{\Theta}{\NSE{\Theta}}{\Ext{\cC}}.
\item $\sdelta_{0}$ is \spEpsAdmissible{\epsilon}{\Theta}{\Ext{\cC}}.
\end{enumerate}
If $\epsilon > 0$, then the following statement is also equivalent:
\begin{enumerate}[start=4]
\item $\sdelta_{0}$ is \spEpsAdmissible{\epsilon}{\NSE{\Theta}}{\Ext{\cC}}.
\item $\sdelta_{0}$ is \spEpsAdmissible{\epsilon}{\NSE{\Theta}}{\NSE{\cC}}.
\end{enumerate}
\end{theorem}
\begin{proof}
Statement (1) is equivalent to
\[
\neg (\exists \delta \in \cC) \text{ $\delta_{0}$ is \EpsDom{\epsilon} by $\delta$.}
\]
By \cref{domthm} and the definition of $\Ext{\cC}$, this is equivalent to both
\[
\neg (\exists \sdelta \in \Ext{\cC}) \text{ $\sdelta_{0}$ is \spEpsDomIn{\epsilon}{\Theta}{\NSE{\Theta}} by $\sdelta$}
\]
and
\[
\neg (\exists \sdelta \in \Ext{\cC}) \text{ $\sdelta_{0}$ \spEpsDom{\epsilon}{\Theta} by $\sdelta$},
\]
hence $(1 \iff 2 \iff 3)$.

Now let $\epsilon >0$. Then the above statements are also equivalent to
\[
\neg (\exists \sdelta \in \Ext{\cC}) \text{ $\sdelta_{0}$ is \spEpsDom{\epsilon}{\NSE{\Theta}} by $\sdelta$},
\]
hence $(1 \iff 4)$. From \cref{lattice}, we see that (5) implies (4).
To see that (1) implies (5), note that,
because $\epsilon$ is standard and $\epsilon > 0$,  (1) is equivalent to
\[
\neg (\exists \delta \in \cC) (\forall \theta\in \Theta)(\Risk(\theta,\delta) \leq \Risk(\theta,\delta_{0})-\epsilon).
\]
By transfer, this statement holds if and only if the following statement holds:
\[
\neg (\exists \Delta \in \NSE{\cC}) (\forall \theta\in \NSE{\Theta})(\sRisk(\theta,\Delta) \leq \sRisk(\theta,\sdelta_{0})-\epsilon).
\]
Again, $\epsilon > 0$ implies
$\sRisk(\theta,\Delta) \not\approx \sRisk(\theta,\sdelta_{0})$ for all $\theta \in \NSE{\Theta}$, hence (5) holds.
\end{proof}

The following corollary for extended admissibility follows immediately.

\begin{theorem}\label{extadEAdm}
Let $\delta_{0} \in \FRRE$
and $\cC \subseteq \FRRE$.
Define $\Ext{\cC} = \{\sdelta : \delta \in \cC\}$.
Then the following statements are equivalent:
\begin{enumerate}
\item $\delta_{0}$ is \ExtAdmissible{\cC}.
\item $\sdelta_{0}$ is \spExtAdmissible{\Theta}{\Ext{\cC}}.
\item $\sdelta_{0}$ is \spExtAdmissible{\NSE{\Theta}}{\Ext{\cC}}.
\item $\sdelta_{0}$ is \spExtAdmissible{\NSE{\Theta}}{\NSE{\cC}}.
\end{enumerate}
\end{theorem}

As in the standard universe, the notion of $\NSE{}$admissibility lead to notions of complete classes.

\begin{definition}\label{nsecomplete}
Let $\cA,\cC \subseteq \sFRRE$.
\begin{enumerate}

\item
$\cA$ is a \defn{\spCompleteSubclass{\cC}} if
for all $\Delta \in \cC \setminus \cA$,
there exists $\Delta' \in \cA$ such that $\Delta$ is \spDom{\Theta} by $\Delta'$.

\item
$\cA$ is an \defn{\spEssCompleteSubclass{\cC}} if
for all $\Delta \in \cC \setminus \cA$,
there exists $\Delta' \in \cA$ such that
$\sRisk(\theta,\Delta') \lessapprox \sRisk(\theta,\Delta)$ for all $\theta\in \Theta$.

\end{enumerate}
\end{definition}

Near-standard essential completeness allows us to enlarge the set of decision procedures amongst which a decision procedure is extended admissible.

\begin{lemma}\label{eCCEA}
Suppose $\cA$ is an \spEssCompleteSubclass{\cC \subseteq \FRRE}.
Then
\spExtAdmissible{\Theta}{\cA} implies \spExtAdmissible{\Theta}{\cC}.
\end{lemma}
\begin{proof}
Let $\Delta_{0} \in \cA$ and suppose $\Delta_{0}$ is not \spExtAdmissible{\Theta}{\cC}.
Then there exists $\Delta \in \cC$ and $\epsilon \in \PosReals$
such that $\sRisk(\theta,\Delta) \le \sRisk(\theta,\Delta_{0}) - \epsilon$ for all $\theta \in \Theta$.
But then by the $\NSE{}$essential completeness of $\cA$,
there exists some $\Delta' \in \cA$,
such that $\sRisk(\theta,\Delta') \lessapprox \sRisk(\theta,\Delta)$ for all $\theta \in \Theta$,
hence $\sRisk(\theta,\Delta') \lessapprox \sRisk(\theta,\Delta_{0}) - \epsilon$ for all $\theta \in \Theta$.
But then $\Delta_{0}$ is not \spEpsAdmissible{\epsilon/2}{\Theta}{\cA} hence not
\spExtAdmissible{\Theta}{\cA}.
\end{proof}

\section{Nonstandard Bayes}
\label{sec: nonstandardbayes}

We now define the nonstandard counterparts to Bayes risk and optimality for the class $\sFRRE$ of internal decision procedures:

\begin{definition}\label{defnhbayes}
Let $\Delta \in \sFRRE$, $\epsilon \in \NSE{\NNReals}$,
and $\cC \subseteq \sFRRE$, and let $\Pi_{0}$ be
a \defn{nonstandard prior}, i.e., an internal probability measure on $(\NSE{\Theta},\NSE{\BorelSets{\Theta}})$.
\begin{enumerate}
\item
The \defn{internal Bayes risk} under $\Pi_{0}$ of $\Delta$
is $\sRisk(\Pi_{0},\Delta)= \int \sRisk(\theta,\Delta) \Pi_{0}(\dee \theta)$.

\item
$\Delta$ is \defn{\seBayesP{\epsilon}{\Pi_{0}}{\cC}} if
$\sRisk(\Pi_{0},\Delta)$ is hyperfinite
and, for all $ \Delta'\in \cC$, we have $\sRisk(\Pi_{0},\Delta) \le \sRisk(\Pi_{0},\Delta') + \epsilon$.

\item
$\Delta$ is \defn{\NSBayesP{\Pi_{0}}{\cC}} if there exists an infinitesimal $\epsilon \in \NSE{\NNReals}$ such that
$\Delta$ is \seBayesP{\epsilon}{\Pi_{0}}{\cC}.

\item
$\Delta$ is \defn{\seBayes{\epsilon}{\cC}} (resp., \defn{\NSBayes{\cC}}) if there exists a nonstandard prior $\Pi$
such that $\Delta$ is \seBayesP{\epsilon}{\Pi}{\cC} (resp., \NSBayesP{\Pi}{\cC}).

\end{enumerate}
\end{definition}

We will sometimes write \defn{\NSBayes{\cC} with respect to $\Pi_{0}$} to mean \NSBayesP{\Pi_{0}}{\cC},
and similarly for \seBayes{\epsilon}{\cC}.
Note that the internal Bayes risk is precisely the extension of the standard Bayes risk. 
Similarly, if we consider the relation $\{ (\delta, \epsilon, \cC) \in \FRRE \times \NNReals \times \PowerSet{\FRRE} : \text{$\delta$ is \eBayes{\epsilon}{\cC}} \}$, then its extension corresponds to
$\{ (\Delta, \epsilon, \cC) \in \sFRRE \times \NSE{\NNReals} \times \NSE{\PowerSet{\FRRE}} : \text{$\Delta$ is \seBayes{\epsilon}{\cC}} \}$.
Note, however, that our definition of ``\seBayes{\epsilon}{\cC}'' allows the set $\cC \subseteq \sFRRE$ to be external, and so it is not simply the transfer of the definition.

Transfer remains a powerful tool for relating the optimality of standard procedures with that of their extensions.
For example, by transfer, $\delta$ is \eBayesP{\epsilon}{\pi}{\cC} if and only if $\sdelta$ is \seBayesP{\epsilon}{\NSE{\pi}}{\NSE{\cC}}. (Recall that $\NSE{\epsilon} = \epsilon$ for a real $\epsilon$, by extension.)
Transfer also yields the following result:

\begin{theorem}\label{eximns}
Let $\cC\subset \FRRE$.
If $\delta_0$ is extended \Bayes{\cC},
then $\sdelta_0$ is \NSBayes{\NSE{\cC}}.
\end{theorem}
\begin{proof}
By hypothesis, the following sentence holds:
\[
(\forall \epsilon \in \PosReals)(\exists \pi \in \ProbMeasures{\Theta})(\forall \delta\in \cC)(\Risk(\pi,\delta_0)\leq \Risk(\pi,\delta)+\epsilon).
\]
By the transfer principle,
\[
(\forall \epsilon \in \NSE{\PosReals})(\exists \pi \in \NSE{\!\!\!\ProbMeasures{\Theta}})(\forall \delta\in \NSE{\cC})(\sRisk(\pi,\sdelta_0)\leq \sRisk(\pi,\sdelta)+\epsilon).
\]
Taking $\epsilon$ to be infinitesimal implies that the internal Bayes risk of $\sdelta_{0}$ is within an infinitesimal of the minimum Bayes risk among $\NSE{\cC}$ with respect to an internal probability measure on $\NSE{\Theta}$,
hence $\sdelta_{0}$ is \NSBayes{\NSE{\cC}}.
\end{proof}

In general,
we would not expect an extension $\sdelta$ to be \seBayesP{0}{\Pi}{\cC}
for a generic nonstandard prior $\Pi$ and class $\cC \subseteq \sFRRE$.
The definition of nonstandard Bayes provides infinitesimal slack, which suffices to yield a precise characterization of extended admissible procedures. The follow result shows that, as in the standard universe,
nonstandard Bayes optimality implies nonstandard extended admissibility.

\begin{theorem}\label{nsBayesImpliesExtAdmissible}
Let $\Delta_{0} \in \sFRRE$, let $\cC \subseteq \sFRRE$, and suppose that $\Delta_{0}$ is \NSBayes{\cC}.
Then $\Delta_{0}$ is \spExtAdmissible{\NSE{\Theta}}{\cC}. 
\end{theorem}
\begin{proof}
Suppose $\Delta_{0}$ is not \spExtAdmissible{\NSE{\Theta}}{\cC}.
Then for some standard $\epsilon \in \PosReals$, $\Delta_{0}$ is \spEpsDom{\epsilon}{\NSE{\Theta}} by some $\Delta \in \cC$, i.e.,
\[
(\forall \theta \in \NSE{\Theta})(\sRisk(\theta,\Delta) \le \sRisk(\theta,\Delta_{0})-\epsilon).
\]
Hence, for every nonstandard prior $\Pi$,
if $\sRisk(\Pi,\Delta)$ is not hyperfinite, then neither is $\sRisk(\Pi,\Delta_{0})$,
and if $\sRisk(\Pi,\Delta)$ is hyperfinite, then
\[
\sRisk(\Pi,\Delta_{0})
&= \int \sRisk(\theta,\Delta_{0}) \Pi(\dee \theta) \\
&\ge \int \sRisk(\theta,\Delta) \Pi(\dee \theta) + \epsilon
= \sRisk(\Pi,\Delta) + \epsilon.
\]
As $\epsilon \in \PosReals$, we conclude that $\Delta_{0}$ cannot be \NSBayesP{\Pi}{\cC}.
As $\Pi$ was arbitrary, $\Delta_{0}$ is not \NSBayes{\cC}.
\end{proof}

\cref{nsBayesImpliesExtAdmissible,extadEAdm} immediately yield the following corollary.
\begin{corollary}
Let $\delta \in \FRRE$ and $\cC \subseteq \FRRE$.
If $\sdelta$ is \NSBayes{\Ext{\cC}}, then $\delta$ is \ExtAdmissible{\cC}.
\end{corollary}

The above result raises several questions:
Are extended admissible decision procedures also nonstandard Bayes?
What is the relationship with admissibility and its nonstandard counterparts?

In this section, we prove that a decision procedure $\delta$ is
extended admissible if and only if $\sdelta$ is nonstandard Bayes.
In later sections, we give several application of this equivalence,
and then consider the relationship with admissibility,
which is far from settled.
It is easy, however, to show that only nonstandard Bayes procedures can
$\NSE{}$dominate other nonstandard Bayes procedures:
To see this, suppose that
$\Delta$ is \NSBayes{\cC \subseteq \sFRRE} with respect to some nonstandard prior $\Pi$
and $\Delta$ is not \spAdmissible{\NSE{\Theta}}{\cC}.

Then $\Delta$ is \spDom{\NSE{\Theta}} by some $\Delta' \in \cC$.
Thus we have $\sRisk(\theta,\Delta') \le \sRisk(\theta,\Delta)$ for all $\theta \in \NSE{\Theta}$.
By \cref{defnhbayes}, we have $\sRisk(\Pi,\Delta)=\int \sRisk(\theta,\Delta)\Pi(\dee\theta)$ hyperfinite.
But then,
$\sRisk(\Pi,\Delta) \lessapprox \sRisk(\Pi,\Delta') = \int \sRisk(\theta,\Delta')\Pi(\dee\theta) \le \sRisk(\Pi,\Delta)$,
hence
$\sRisk(\Pi,\Delta) \approx \sRisk(\Pi,\Delta')$,
hence $\Delta'$ is \NSBayesP{\Pi}{\cC}.
This proves a nonstandard version of a
well-known standard result stating that every unique Bayes procedure is admissible \citep[][\S2.3 Thm.~1]{Ferguson}:
\begin{theorem}\label{domisbayes}
Suppose $\Delta$ is \NSBayes{\cC \subseteq \sFRRE} with respect to a nonstandard prior $\Pi$.
If $\Delta$ is \spDom{\NSE{\Theta}} by $\Delta' \in \cC$,
then $\Delta'$ is \NSBayesP{\Pi}{\cC}.
Therefore, if $\sRisk(\theta,\Delta') \approx \sRisk(\theta,\Delta)$
for all $\theta \in \NSE{\Theta}$ and
for all $\Delta' \in \cC$ such that $\Delta'$ is \NSBayesP{\Pi}{\cC},
then $\Delta$ is \spAdmissible{\NSE{\Theta}}{\cC}.
\end{theorem}
\begin{proof}
The first statement follows from the logic in the preceding paragraph.
Now suppose that $\Delta$ is \spDom{\NSE{\Theta}} by some $\Delta' \in \cC$.
Then $\Delta'$ is \NSBayesP{\Pi}{\cC}.
But then, by hypothesis, its risk function is equivalent, up to an infinitesimal, to that of $\Delta$, a contradiction.
\end{proof}

\subsection{Hyperdiscretized risk set}

In a statistical decision problem with a finite parameter space,
one can use a separating hyperplane argument to show that every admissible decision procedure is Bayes (see, e.g., \citep[][\S2.10 Thm.~1]{Ferguson}).
In order to prove our main theorem, we will proceed along similar lines, but with the aid of extension, transfer, and saturation.

When relating extended admissibility and Bayes optimality for a subclass $\cC \subseteq \FRRE$,
the set of all risk functions $\Risk_\delta$, for $\delta \in \cC$, is a key structure.
On a finite parameter space, the risk set for $\FRRE$ is a convex subset of a finite-dimensional vector space over $\Reals$.
When the parameter space is not finite, one must grapple with infinite dimensional function spaces.
However, in a sufficiently saturated nonstandard model,
there exists an internal set $\TTheta \subset \NSE{\Theta}$ that is hyperfinite and contains $\Theta$.
While the risk at all points in $\TTheta$ does not suffice to characterize an arbitrary element of $\sFRRE$,
it suffices to study the optimality of extensions of standard
decision procedure \emph{relative to other extensions}.
Because $\TTheta$ is hyperfinite,
the corresponding risk set is a convex subset of a hyperfinite-dimensional vector space over $\HReals$.

Let $\STK \in \NSE{\Nats}$ be the internal cardinality of $\TTheta$
and let $\TTheta = \{\ttheta_1,\dots,\ttheta_{\STK}\}$. Recall that $\DRiskSpace$ denotes
the set of (internal) functions from $\TTheta$ to $\HReals$.  For an element $x \in \DRiskSpace$,
we will write $x_k$ for $x(k)$.

\begin{definition}\label{disriskset}
The \defn{hyperdiscretized risk set induced by $D \subseteq \sFRRE$ }
is the set
\[
\RS{D}=\{ x \in \DRiskSpace: (\exists \Delta \in D)\,(\forall k \leq \STK ) \, x_{k} = \sRisk(\ttheta_{k},\Delta) \} \subset \DRiskSpace.
\]
\end{definition}

\begin{lemma}\label{rsconvex}
Let $D \subseteq \sFRRE$ be an internal convex set.
Then $\RS{D}$ is an internal convex set.
\end{lemma}
\begin{proof}
$\RS{D}$ is internal by the internal definition principle and the fact that $D$ is internal.
In order to demonstrate convexity, pick $p\in \NSE{[0,1]}$, and let $x,y\in \RS{D}$.
Then there exist $\Delta_1,\Delta_2\in D$ such that
$x_{k}=\sRisk(\ttheta_{k},\Delta_1)$ and $y_{k}=\sRisk(\ttheta_{k},\Delta_2)$ for all $k\leq \STK$.
Because $D$ is convex, $p\Delta_1+(1-p)\Delta_2\in D$.
But $px_k+(1-p)y_k=\sRisk(\ttheta_{k},p\Delta_1+(1-p)\Delta_2)$ for all $k\leq \STK$,
and so $\RS{D}$ is convex.
\end{proof}

\begin{definition}\label{nrest}
For every $\cA\subset \FRRE$, define $\Ext{\cA}=\{\sdelta: \delta\in \cA\}$,
and, for every $\cC \subseteq \sFRRE$,
let
\[
\Fincomb{\cC}=\bigcup_{D \in \FiniteSubsets{\cC}} \sconv(D)
\]
be the set of all finite $\NSE{}$convex combinations of $\sdelta \in \cC$.
\end{definition}

Note that $\Ext{\cA}\subset \NSE{\cA}$ is an external set unless $\cA$ is finite.
Let $\delta_1,\delta_2\in \FiniteRiskEstimators$ and let $p\in \NSE{[0,1]}$.
Then $p\sdelta_1+(1-p)\sdelta_2\in \Ext{\FRREFS}$ if $p\in [0,1]$.
However, $p\sdelta_1+(1-p)\sdelta_2\in \FSE$ for all $p\in \NSE{[0,1]}$.
It is easy to see that $\Fincomb{\Ext{\FRREFS}}=\FSE$.
Thus, we have $\ExtFiniteRiskEstimators \subset \Ext{\FRREFS} \subset \Fincomb{\Ext{\FRREFS}}=\FSE \subset \sFRREFS$.

\begin{lemma}
For any $\cC \subseteq \sFRRE$, $\Fincomb{\cC}$ is a convex set containing $\cC$.
\end{lemma}
\begin{proof}
Pick an $\cC\subseteq \sFRRE$.
Clearly $\Fincomb{\cC}\supset \cC$.
It remains to show that $\Fincomb{\cC}$ is a convex set.
Pick two elements $\Delta_1,\Delta_2\in \Fincomb{\cC}$.
Then there exist $D_1,D_2\in \FiniteSubsets{\cC}$ such that
$\Delta_1\in \sconv(D_1)$ and $\Delta_2\in \sconv(D_2)$.
Let $p\in \NSE{[0,1]}$.
It is easy to see that $p\Delta_1+(1-p)\Delta_2\in \sconv(D_1\cup D_2)$.
\end{proof}

\begin{lemma}\label{newLecomplete3}
$\Ext{\FRREFS}$ is an \spEssCompleteSubclass{\FSE}.
\end{lemma}
\begin{proof}
Let $\Delta \in \FSE$.
Then $\Delta=\sum_{i=1}^{n} p_i \sdelta_i$
for some $n \in \Nats$, $\delta_1,\dotsc,\delta_n \in \FiniteRiskEstimators$, and
 $p_1,\dotsc,p_n \in \NSE{\NNReals}$, $\sum_{i=1}^{n} p_i = 1$.
Define $\Delta_0 = \sum_{i=1}^{n} \SP{\!p_i} \sdelta_i$
and let $\theta \in \Theta$.
For all $i \le n$,
$p_i \sRisk(\theta,\sdelta_{i}) \approx \SP{\!p_i} \sRisk(\theta,\sdelta_{i})$
because $\sRisk(\theta,\sdelta_{i})$ is finite,
so  $\sRisk(\theta,\Delta) \approx \sRisk(\theta,\Delta_0)$.
By \cref{nsecomplete}, $\Ext{\FRREFS}$ is an \spEssCompleteSubclass{\FSE}.
\end{proof}{

Having defined the (hyperdiscretized) risk set, we now describe a set whose intersection with the risk set captures the notion of $\frac 1 n$-$\NSE{}$domination, for some standard $n \in \Nats$.
In that vein, for $\Delta \in \sFRRE$, define the \defn{$\frac 1 n$-quantant}
\[
Q(\Delta)_{n}=\{ x \in \DRiskSpace: (\forall k\leq \STK )(x_k \leq \sRisk(\ttheta_k,\Delta)-\frac{1}{n})\}, \qquad n \in \NSE{\Nats}.
\]

\begin{lemma}\label{Qconvex}
Fix $\Delta\in \sFRRE$. The set $Q(\Delta)_{n}$ is internal and convex and
$Q(\Delta)_{m}\subset Q(\Delta)_{n}$ for every $m<n$.
\end{lemma}

\begin{proof}
By the internal definition principle, $Q(\Delta)_{n}$ is internal.
Let $x,y$ be two points in $Q(\Delta)_{n}$, let $p\in \NSE{[0,1]}$, and pick a coordinate $k$. Then
\[
px_{k}+(1-p)y_{k}\leq p(\sRisk(\ttheta_k,\Delta)-\frac{1}{n})+(1-p)(\sRisk(\ttheta_k,\Delta)-\frac{1}{n})=(\sRisk(\ttheta_k,\Delta)-\frac{1}{n}).
\]
Thus the set is convex.
The second statement is obvious.
\end{proof}

The following is then immediate from definitions.

\begin{lemma}\label{exempty}
Let $\cC \subseteq \sFRRE$ and $n \in \Nats$. Then
$\Delta$ is \spEpsAdmissible{\frac 1 n}{\TTheta}{\cC}
if and only if
$Q(\Delta)_{n}\cap \GENRS=\emptyset$.
\end{lemma}

\subsection{Nonstandard complete class theorems}

\begin{lemma}
\label{hyperplanetoprior}
Let $\Delta \in \sFRRE$ and nonempty $D \subseteq \sFRRE$,
and suppose there exists a nonzero vector $\Pi \in \DRiskSpace$
such that $\IP{\Pi}{x} \leq \IP{\Pi}{s}$
for all $x\in \bigcup_{n\in \Nats}Q(\Delta)_{n}$
and $s\in \cS^{D}$.
Then the normalized vector $\Pi / \|\Pi\|_1$ induces an internal probability measure $\pi$ on $\NSE{\Theta}$ concentrating on
$\TTheta$, and $\Delta$ is \NSBayesP{\pi}{D}.
\end{lemma}
\begin{proof}
We first establish that $\Pi(k) \geq 0$ for all $k$.
Suppose otherwise, i.e., $\Pi(k_0)<0$ for some $k_{0}$.
Then we can pick a point $x_{0}$ in $\bigcup_{n\in \Nats}Q(\Delta)_{n}$ whose $k_{0}$-th coordinate is arbitrarily large and negative,
causing $\IP{\Pi}{x_{0}}$ to be arbitrary large, a contradiction because $\IP{\Pi}{s}$ is hyperfinite for all $s \in \cS^{D}$.
Hence, all coordinates of $\Pi$ must be nonnegative.

Define $\pi \in \DRiskSpace$
 by $\pi = \Pi / \| \Pi \|_1 $.
Because $\Pi \neq 0$ and $\Pi \ge 0$, we have $\pi \ge 0$ and $\| \pi \|_1 = 1$.
Therefore, $\pi$ specifies an internal probability measure on $(\NSE{\Theta},\NSE{\BorelSets{\Theta}})$, concentrating on $\TTheta$, 
and assigning probability $\pi(k)$ to $\ttheta_k$ for every $k \leq J_{\Theta}$.
Because $\| \Pi \|_1 > 0$,
it still holds that $\IP{\pi}{x} \leq \IP{\pi}{s}$ for all $x\in \bigcup_{n\in \Nats}Q(\Delta)_{n}$ and $s\in \cS^{D}$.

Let $s\in \cS^{D}$.
Then
$\SP{(\sum_{k\in \STK}\pi_{k}(\sRisk(\ttheta_k,\Delta)-\frac{1}{n}))} \leq \SP{(\sum_{k\in \STK}\pi_{k}s_{k})}$
for every $n\in \Nats$.
The l.h.s.\ is simply
$\SP{(-\frac {1}{n} + \sum_{k\in \STK}\pi_{k} \sRisk(\ttheta_k,\Delta))}$,
and the limit of this expression as $n \to \infty$ is
$\SP{(\sum_{k\in \STK}\pi_{k} \sRisk(\ttheta_k,\Delta))}$.
Hence,
$\sum_{k\in \STK}\pi_{k}(\sRisk(\ttheta_k,\Delta) \lessapprox \sum_{k\in \STK}\pi_{k} s_{k} $.
This shows that $\Delta$ is \NSBayesP{\pi}{D}.
\end{proof}

The previous result shows that if a nontrivial hyperplane separates the risk set from every $\frac 1 n$-quantant, for $n \in \Nats$,
then the corresponding procedure is nonstandard Bayes.
In order to prove our main theorem, we require a nonstandard version of the hyperplane separation theorem, which we give here.
For $a,b \in \Reals^k$ for some finite $k$,
let $\IP{a}{b}$ denote the inner product.
We begin by stating the standard hyperplane separation theorem:

\begin{theorem}[Hyperplane separation theorem]
For any $k\in \Naturals$, let $S_{1}$ and $S_{2}$ be two disjoint convex subsets of $\Reals^{k}$, then there exists $w\in \Reals^{k}\setminus \{\vzero\}$ such that, for all $p_1\in S_{1}$ and $p_{2}\in S_{2}$, we have $\IP{w}{p_{1}} \geq \IP{w}{p_{2}}$.
\end{theorem}

Using a suitable encoding of this theorem in first-order logic,
the transfer principle yields a hyperfinite version:

\begin{theorem}\label{hyperplanethm}
Fix any $K\in \NSE{\Naturals}$. If $S_{1},S_{2}$ are two disjoint internal convex subsets of $I({\HReals}^{K})$, then there exists $W\in \IFuncs{\HReals}{K} \setminus \{\vzero\}$ such that, for all $P_1\in S_{1}$ and $P_{2}\in S_{2}$, we have
$\IP{W}{P_{1}} \geq \IP{W}{P_{2}}$.
\end{theorem}

See \cref{vectorspacessec} for a proof.

Recall that our nonstandard model is $\kappa$-saturated for some infinite $\kappa$.

\begin{theorem}\label{extCC}
Let $\cC \subseteq \Ext{\FRRE}$ be a (necessarily finite or external) set with cardinality less than $\kappa$,
and suppose that $\cC$ is a \spEssCompleteSubclass{\Fincomb{\cC}}.
Let $\Delta_{0} \in \sFRRE$ and suppose $\Delta_{0}$ is \spExtAdmissible{\Theta}{\cC}.
Then, for every hyperfinite set $T \subseteq \NSE{\Theta}$ containing $\Theta$,
$\Delta_{0}$ is \NSBayes{\Fincomb{\cC}} with respect to some nonstandard prior concentrating on $T$.
\end{theorem}

\begin{proof}
Without loss of generality we may take $T = \TTheta$.
By \cref{eCCEA} and the fact that $\cC$ is an \spEssCompleteSubclass{\Fincomb{\cC}}, $\Delta_0$ is \spExtAdmissible{\Theta}{\Fincomb{\cC}}.
By \cref{lattice}, $\Delta_{0}$ is \spEpsAdmissible{\frac 1 n}{\TTheta}{\Fincomb{\cC}} for every $n\in \Nats$.
Hence, by \cref{exempty}, $Q(\Delta_{0})_{n}\cap \RS{\Fincomb{\cC}} = \emptyset$ for all $n\in \Nats$.

By the definition of $\Fincomb{\cC}$,
 we have $Q(\Delta_{0})_{n}\cap \cS^{\sconv(D)}=\emptyset$ for every $D \in \FiniteSubsets{\cC}$.
By \cref{Qconvex,rsconvex},
$\cS^{\sconv(D)}$ and $Q(\Delta_{0})_{n}$ are both internal convex sets,
hence, by \cref{hyperplanethm}, there is a nontrivial hyperplane $\Pi_{n}^{D} \in \DRiskSpace$ that separates them.

For every $D \in \FiniteSubsets{\cC}$
and $n \in \Naturals$,
let $\phi_{n}^{D}(\Pi)$ be the formula
\[
( \Pi \in \DRiskSpace)\land \,
(\Pi \neq \vzero \land
(\forall x\in Q(\Delta_{0})_{n}) \,
(\forall s\in \cS^{\sconv(D)}) \,
\IP{\Pi}{x} \leq \IP{\Pi}{s} ),
\]
and let $\FSC=\{\phi_{n}^{D}(\Pi): n \in \Naturals,\, D \in \FiniteSubsets{\cC} \}$.
By the above argument
and the fact that $\FiniteSubsets{\cC}$ is closed under taking finite unions
and the sets $Q(\Delta_{0})_n$, for $n \in \Nats$, are nested,
$\FSC$ is finitely satisfiable.
Note that $\FSC$ has cardinality no more than $\kappa$,
yet our nonstandard extension is $\kappa$-saturated by hypothesis.
Therefore, by the saturation principle,
there exists a nontrivial hyperplane $\Pi$ satisfying every sentence in $\FSC$ simultaneously.
That is, there exists $\Pi \in \DRiskSpace$ such that $\Pi \neq \vzero$ and,
for all $x \in \bigcup_{n\in \Nats}Q(\Delta_{0})_{n}$
and for all $s\in \bigcup_{D \in \FiniteSubsets{\cC} } \cS^{\sconv(D)} = \RS{\Fincomb{\cC}}$,
we have  $\IP{\Pi}{x} \leq \IP{\Pi}{s}$.

Hence, by \cref{hyperplanetoprior},
the normalized vector $\Pi / \|\Pi\|_1$ is well-defined and induces a probability measure $\pi$ on $\NSE{\Theta}$ concentrating on
$\TTheta$, and $\Delta_{0}$ is \NSBayesP{\pi}{\Fincomb{\cC}}.
\end{proof}

\begin{theorem}\label{fsecomplete}
For $\delta_{0} \in \FRRE$, the following are equivalent statements:
\begin{enumerate}
\item $\delta_{0}$ is \ExtAdmissible{\FRREFS}.
\item $\sdelta_{0}$ is \NSBayes{\Ext{\FRREFS}}.
\item $\sdelta_{0}$ is \NSBayes{\FSE}.
\end{enumerate}
If \cref{assumptioncv} also holds, then the following statements are also equivalent:
\begin{enumerate}[start=4]
\item $\delta_{0}$ is \ExtAdmissible{\FiniteRiskEstimators}.
\item $\sdelta_{0}$ is \NSBayes{\Ext{\FiniteRiskEstimators}}.
\end{enumerate}
Moreover,
statements (2), (3), and (5) can be taken to assert that,
for all hyperfinite sets $T \subseteq \NSE{\Theta}$ containing $\Theta$,
Bayes optimality holds with respect to some nonstandard prior concentrating on $T$.
\end{theorem}

\begin{proof}
From (1) and \cref{extadEAdm},
$\sdelta_{0}$ is \spExtAdmissible{\Theta}{\Ext{\FRREFS}}.
It follows from \cref{newLecomplete3,extCC} that,
for all hyperfinite sets $T \subseteq \NSE{\Theta}$ containing $\Theta$,
$\sdelta_{0}$ is \NSBayes{\FSE} with respect to some nonstandard prior $\pi$ concentrating on $T$.
Hence (3) holds and (2) follows trivially.

From (2) and \cref{nsBayesImpliesExtAdmissible},
it follows that
$\sdelta_{0}$ is \spExtAdmissible{\NSE{\Theta}}{\Ext{\FRREFS}}.
Then (1) follows from \cref{extadEAdm}.

It is the case that (1) implies (4) by \cref{lattice}, and the other direction follows from \cref{assumptioncv}, \cref{FREconvexECC}, and \cref{esscomplete}.
Similarly, (2) implies (5).
Finally, from (5) and \cref{nsBayesImpliesExtAdmissible},
it follows that
$\sdelta_{0}$ is \spExtAdmissible{\NSE{\Theta}}{\Ext{\FiniteRiskEstimators}}.
Then (4) follows from \cref{extadEAdm}.
\end{proof}

It follows immediately that the class of extended admissible procedures is a complete class if and only if the class of procedures whose extensions are nonstandard Bayes are a complete class.

\section{Applications to compact statistical decision problems}
\label{sec: CCs under compactness and RC}

In this section,
we use our nonstandard theory to prove
that,
under the additional hypotheses that $\Theta$ is compact (and thus normal) and all risk functions are continuous,
the class of extended admissible procedures is precisely the class of Bayes procedures. The strength of our result lies in the absence of any additional assumptions on the loss or model.\footnote{
In \cref{sec: nonstandardbayes}, the Hausdorff condition can be sidestepped by adopting the discrete topology.
Unless $\Theta$ is finite, however, $\Theta$ will not be compact under the discrete topology.
Thus, the topological hypotheses in this section
not only determine the space of priors, but also
restrict the set of decision problems to which the theory applies.
}

Assume $\sdelta$ is nonstandard Bayes with respect to some nonstandard prior $\pi$ on $\NSE{\Theta}$.
In this section, we will construct a standard probability measure $\pi_{p}$ on $\Theta$ from $\pi$ in such a way that the internal risk of $\sdelta$ under $\pi$ is infinitesimally close to the risk of $\delta$ under $\pi_{p}$.  This then implies that $\pi$ is Bayes with respect to $\pi_{p}$, and yields a standard characterization of extended admissible procedures.

Extension allows us to associate an internal probability measure $\NSE{\pi}$ to every standard probability measure $\pi$.
The next theorem describes a reverse process via Loeb measures.

\begin{lemma}[{\citep[][Thm.~13.4.1]{NDV}}]\label{pushdown}
Let $\gY$ be a compact Hausdorff space equipped with Borel $\sigma$-algebra $\BorelSets {\gY}$,
let $\nu$ be an internal probability measure defined on $(\NSE{\gY}, \NSE{\BorelSets{\gY}})$,
and let $\cC=\{C\subset \gY: \ST^{-1}(C) \in \LoebAlgebraX{\NSE{\BorelSets{\gY}}}{\nu} \}$.
Define a probability measure $\nu_{p}$ on the sets $\cC$ by $\nu_{p}(C)=\Loeb{\nu}(\ST^{-1}(C))$. Then $(\gY,\cC,\nu_{p})$ is the completion of a regular Borel probability space.
\end{lemma}

Note that $\ST^{-1}(E)$ is Loeb measurable for all $E\in \BorelSets{\gY}$ by \cref{LRunivLoeb,nscompact}.

\begin{definition}\label{defpushdown}
The probability measure $\nu_{p} : \cC \to [0,1]$ in \cref{pushdown} is called the \defn{pushdown} of the internal probability measure $\nu$.
\end{definition}

\begin{example}
If a nonstandard prior concentrates on finitely many points in $\NS{\NSE{\Theta}}$,
then its pushdown concentrates on the standard parts of those points, hence is a standard measure with support on a finite set.
\end{example}

\begin{example}
Suppose  $S=[K^{-1},2 K^{-1}, \dots, 1-K^{-1}, 1]$ for some nonstandard natural $K \in \NSE{\Nats} \setminus \Nats $.
Define an internal probability measure $\pi$ on $\NSE{[0,1]}$ by $\pi \{s\}=K^{-1}$ for all $s\in S$,
and let $\pi_{p}$ be its pushdown.
Then $\pi_{p}$ is Lebesgue measure on $[0,1]$.
\end{example}

The following lemma establishes a close link between Loeb integration and integration with respect to the pushdown measure.

\begin{lemma}\label{newpushdownint}
Let $\gY$ be a compact Hausdorff space equipped with Borel $\sigma$-algebra $\BorelSets {\gY}$,
let $\nu$ be an internal probability measure on $(\NSE{\gY},\NSE{\BorelSets {\gY}})$,
let $\nu_{p}$ be the pushdown of $\nu$, and
let $f : \gY \to \Reals$ be a bounded measurable function. Define $g : \NSE{\gY} \to \Reals$ by $g(s) = f(\SP{s})$.
Then we have $\int f \dee \nu_{p}=\int g \,\dee \Loeb{\nu}$.
\end{lemma}

\begin{proof}
For every $n\in \Nats$ and $k \in \Ints$,
define $F_{n,k} = f^{-1}([\frac{k}{n},\frac{k+1}{n}))$ and $G_{n,k}=g^{-1}(\NSE{[\frac{k}{n},\frac{k+1}{n})})$.
As $f$ is bounded, the collection $\cF_{n} = \{F_{n,k}: k\in \Ints\} \setminus \{\emptyset\}$ forms a finite partition of $\gY$, and similarly for
$\cG_{n} = \{G_{n,k}:k\in \Ints\} \setminus \{\emptyset\}$ and $\NSE{\gY}$.
For every $n \in \Nats$,
define $\hat{f}_{n} : \gY \to \Reals$ and $\hat{g}_{n} : \NSE{\gY} \to \Reals$ by putting
$\hat{f}_{n} = \frac{k}{n}$ on $F_{n,k}$ and
$\hat{g}_{n} = \frac{k}{n}$ on $G_{n,k}$ for every $k \in \Ints$.
Thus $\hat{f}_{n}$ (resp., $\hat{g}_{n}$) is a simple (resp., $\NSE{}$simple) function on the partition $\cF_{n}$ (resp., $\cG_{n}$).
By construction $\hat f_{n} \le f < \hat{f}_{n} + \frac 1 n$ and $\hat{g}_{n} \le g < \hat{g}_{n} + \frac 1 n$.
Note that
$
G_{n,k}=\ST^{-1}(F_{n,k})
$
for every $n\in \Nats$ and $k\in \Ints$.
Moreover, $\gY$ is even regular Hausdorff,
hence \cref{LRunivLoeb} implies that
$G_{n,k}$ is $\Loeb{\nu}$-measurable.
It follows that $\int f \dee \nu_{p} = \lim_{n\to\infty} \int \hat{f}_{n} \dee \nu_{p}$ and $\int g \dee \Loeb{\nu} = \lim_{n \to \infty} \int \hat{g}_{n} \dee \Loeb{\nu}$.
Moreover, by \cref{pushdown}, we have $\Loeb{\nu}(G_{n,k})=\nu_{p}(F_{n,k})$ for every $n\in \Nats$ and $k\in \Ints$.
Thus, for every $n\in \Nats$ and $k\in \Ints$, we have $\int \hat f_n \dee \nu_{p} = \int \hat g_n \dee \Loeb{\nu}$. Hence we have $\int g \,\dee \Loeb{\nu}=\int f \dee \nu_{p}$, completing the proof.
\end{proof}

In order to control the difference between the internal and standard Bayes risks under a nonstandard prior $\pi$ and its pushdown $\pi_{p}$,
it will suffice to require that risk functions be continuous. (Recall that we quoted results listing natural conditions that imply continuous risk in \cref{bdrisk,bdlossbdr}.)

\begin{condition}{RC}[risk continuity]
\label{assumptionrc}
$\Risk(\argdot,\delta)$ is continuous on $\Theta$, for all $\delta\in \FRRE$.
\end{condition}

In order to understand the nonstandard implications of this regularity condition,
we introduce the following definition from nonstandard analysis.

\begin{definition}\label{Scontinuous}
Let $\gX$ and $\gY$ be topological spaces.
A function $f: \NSE{\gX} \to \NSE{\gY}$ is S-continuous at $x \in \NSE{\gX}$ if $f(y) \approx f(x)$ for all $y \approx x$.
\end{definition}

A fundamental result in nonstandard analysis links continuity and S-continuity:

\begin{lemma}\label{contScont}
Let $\gX$ and $\gY$ be Hausdorff spaces, where $\gY$ is also locally compact, and let $D \subseteq \gX$.
If a function $f : \gX \to \gY$ is continuous on $D$
then  its extension $\NSE{\!f}$ is $\NS{\NSE{\gY}}$-valued and S-continuous on $\NS{\NSE{D}}$.
\end{lemma}

See \cref{sec:hytopo} for a proof of this classical result.
We are now at the place to establish the correspondence between internal Bayes risk and standard Bayes risk.
The proof relies on the following technical lemma.

\begin{lemma}[{\citep[][Cor.~4.6.1]{NSAA97}}]\label{bintegration}
Suppose $(\Omega,\cF,P)$ is an internal probability space, and $F:\Omega\to {\HReals}$ is an internal $P$-integrable function such that $^{\circ}F$ exists everywhere.
Then ${^{\circ}F}$ is integrable with respect to $\Loeb{P}$ and $\int F \dee P \approx \int {^{\circ}F} \dee\Loeb{P}$.
\end{lemma}

\begin{lemma}\label{newpushdownthm}
Suppose $\Theta$ is compact Hausdorff and \cref{assumptionrc} holds.
Let $\pi$ be an internal distribution on $\NSE{\Theta}$ and let $\pi_{p} : \cC \to [0,1]$ be its pushdown.
Let $\delta_0 \in \FRRE$ be a standard decision procedure.
If $\sRisk(\argdot, \sdelta_0)$ is $\pi$-integrable then $\Risk(\argdot,\delta_{0})$ is a $\pi_p$-integrable function and $\Risk(\pi_p,\delta_0)\approx \sRisk(\pi,\sdelta_0)$,
i.e., the Bayes risk under $\pi_{p}$ of $\delta_{0}$ is within an infinitesimal of the nonstandard Bayes risk under $\pi$ of $\sdelta_{0}$.
\end{lemma}

\begin{proof}
Because $\Theta$ is compact Hausdorff,
$\SP{t}$ exists for all $t \in \NSE{\Theta}$
and
\cref{pushdown} implies $\pi_{p}$ is a probability measure on $(\Theta,\cC)$, where $\cC$ is the $\pi_{p}$-completion of $\BorelSets{\Theta}$.
By \cref{assumptionrc,contScont},
for all $t \in \NSE{\Theta}$,
we have
\[
 \sRisk(t,\sdelta_0) \approx \sRisk(\SP{t}, \sdelta_0) = \Risk(\SP{t},\delta_0).
\]
Hence $\SP({\sRisk(t,\delta_0)}) = \Risk(\SP{t},\delta_0)$ exists for all $t \in \NSE{\Theta}$.
As $\sRisk(\argdot, \sdelta_0)$ is $\pi$-integrable, by \cref{bintegration},
we know that $\SP{(\sRisk(\argdot,\sdelta_0))}$ is $\Loeb{\pi}$-integrable and
\[\label{riskintegral}
\int \sRisk(t,\sdelta_0) \pi(\dee t)
\approx \int \SP{(\sRisk(t,\sdelta_{0}))} \Loeb{\pi}(\dee t)
= \int \sRisk(\SP{t},\sdelta_0) \Loeb{\pi}(\dee t)
.
\]
By \cref{assumptionrc} and the fact that $\Theta$ is compact,
it follows that $\Risk(\argdot,\delta_{0})$ is bounded.
Thus, by \cref{newpushdownint},
$\int \sRisk(\SP{t},\sdelta_0) \Loeb{\pi}(\dee t) = \int \Risk(\theta,\delta_{0}) \pi_{p}(\dee \theta)$,
completing the proof.
\end{proof}

\begin{lemma}\label{nsToBayes}
Suppose $\Theta$ is compact Hausdorff and \cref{assumptionrc} holds.
Let $\delta_{0} \in \FRRE$ and $\cC \subseteq \FRRE$.
If $\sdelta_{0}$ is \NSBayes{\Ext{\cC}},
then $\delta_{0}$ is \Bayes{\cC}.
\end{lemma}
\begin{proof}
By \cref{fsecomplete}, we may assume that $\sdelta_{0}$ is \NSBayes{\Ext{\cC}}
with respect to a nonstandard prior $\pi$ that concentrates on some hyperfinite set $T$.
Let $\delta \in \cC$. Then $\sdelta \in \Ext{\cC}$, hence
$\sRisk(\pi,\sdelta_{0}) \lessapprox \sRisk(\pi,\sdelta)$.
Let $\pi_{p}$ denote the pushdown of $\pi$.
As $\Theta$ is compact Hausdorff, we know that $\pi_{p}$ is a probability measure.
As $\pi$ concentrates on the hyperfinite set $T$,
we know that $\sRisk(\argdot, \sdelta_0)$ and $\sRisk(\argdot, \sdelta)$ are $\pi$-integrable.
By \cref{newpushdownthm}, we have $\Risk(\pi_{p},\delta_{0}) \approx \sRisk(\pi,\sdelta_0)$ and $\Risk(\pi_{p},\delta) \approx \sRisk(\pi,\sdelta)$.
Then \cref{standardnums}.2 implies that $\Risk(\pi_{p},\delta_{0}) \leq \Risk(\pi_{p},\delta)$.
As our choice of $\delta$ was arbitrary, $\delta_{0}$ is \BayesP{\pi_{p}}{\cC}.
\end{proof}

\begin{theorem}\label{stbayes}
Suppose $\Theta$ is compact Hausdorff and \cref{assumptionrc} holds.
For $\delta_{0} \in \FRRE$, the following statements are equivalent:
\begin{enumerate}
\item $\delta_{0}$ is \ExtAdmissible{\FRREFS}.
\item $\delta_{0}$ is extended \Bayes{\FRREFS}.
\item $\delta_{0}$ is \Bayes{\FRREFS}.
\end{enumerate}
If \cref{assumptioncv} also holds, then the equivalence extends to these statements with $\FiniteRiskEstimators$ in place of $\FRREFS$.
\end{theorem}

\begin{proof}
Suppose (1) holds. Then by \cref{fsecomplete},
$\sdelta_0$ is \NSBayes{\Ext{\FRREFS}}.
Then (3) follows from \cref{nsToBayes}.
The reverse implications follows from \cref{BayesImpliesExtAdm}.

The statements with $\FRREFS$ imply those for $\FiniteRiskEstimators \subseteq \FRREFS$ trivially.
When \cref{assumptioncv} holds, we have \cref{FREconvexECC}.
Hence, the reverse implications follows from \cref{esscomplete} and \cref{EssCompleteBayes}.
\end{proof}

We conclude this section with a strengthening of \cref{fsecomplete}, showing that infinitesimal $\NSE{}$Bayes risk yields zero $\NSE{}$Bayes risk, and that a procedure is optimal among all extensions if and only if it optimal among all internal estimators:

\begin{corollary}\label{fsestronger}
Suppose $\Theta$ is compact Hausdorff and \cref{assumptionrc} holds.
For $\delta_{0} \in \FRRE$, the following statements are equivalent:
\begin{enumerate}
\item $\delta_{0}$ is \ExtAdmissible{\FRREFS}.
\item $\sdelta_{0}$ is \NSBayes{\sFRREFS}.
\item $\sdelta_{0}$ is \seBayes{0}{\sFRREFS}.
\end{enumerate}
Moreover, the equivalence extends to these statements with $\Ext{\FRREFS}$ in place of $\sFRREFS$.
If \cref{assumptioncv} also holds, the equivalence extends to
these statement with $\FiniteRiskEstimators/\Ext{\FiniteRiskEstimators}/\sFiniteRiskEstimators$
in place of $\FRREFS/\Ext{\FRREFS}/\sFRREFS$.

\end{corollary}
\begin{proof}
Statement (1) implies that $\delta_{0}$ is \Bayes{\FRREFS} by \cref{stbayes}.
This implies (3) by transfer, (3) implies (2) by definition, and (2) implies (1) by \cref{fsecomplete}.

Statements (2) and (3) with $\sFRREFS$ imply their counterparts with $\Ext{\FRREFS}$ in place of $\sFRREFS$, trivially.
Statement (3) with $\Ext{\FRREFS}$ implies (2) with $\Ext{\FRREFS}$ which implies (1) by \cref{fsecomplete}.

The additional equivalences under \cref{assumptioncv} follow by the same logic as above and in the proof of \cref{fsecomplete}.
\end{proof}

\section{Admissibility of nonstandard Bayes procedures}
\label{sec: admissibility}

Heretofore, we have focused on the connection between extended admissibility and nonstandard Bayes optimality.
In this section, we shift our focus to the admissibility of decision procedures whose extensions are nonstandard Bayes.
In all but the final result of this section, \emph{we will assume that $\Theta$ is a metric space} and write $d$ for the metric.

On finite parameter spaces with bounded loss, it is known that Bayes procedures with respect to priors assigning positive mass to every state are admissible. Similarly, when risk functions are continuous, Bayes procedures with respect to priors with full support are admissible. We can establish analogues of these result on general parameter spaces by a suitable nonstandard relaxation of a standard prior having full support.

\begin{definition}
For $x,y \in \HReals$, write $x \gg y$
when $\gamma \, x > y$ for all $\gamma \in \PosReals$.
\end{definition}

\begin{definition}
Let $X$ be a metric space with metric $d$, and let $\epsilon \in \NSE{\PosReals}$.
An internal probability measure $\pi$ on $\NSE{\Theta}$ is \defn{$\epsilon$-regular}
if, for every $\theta_{0} \in \Theta$ and non-infinitesimal $r>0$,
we have $\pi(\{ t \in \NSE{\Theta}: \NSE{d}(t,\theta_0) < r \}) \gg \epsilon$.
\end{definition}

The following result establishes $\NSE{}$admissibility from $\NSE{}$Bayes optimality under conditions analogues to full support and continuity of the risk function.

\begin{lemma} \label{ensblyth}
Suppose $\Theta$ is a metric space. Let $\epsilon\in \NSE{\PosReals}$, $\Delta_{0}\in \sFRRE$, and $\cC\subseteq \sFRRE$,
and suppose $\sRisk(\argdot,\Delta)$ is S-continuous on $\NS{\NSE{\Theta}}$ for all $\Delta \in \cC\cup\{\Delta_0\}$.
If $\Delta_{0}$ is \seBayes{\epsilon}{\cC}
with respect to an $\epsilon$-regular nonstandard prior $\pi$,
then $\Delta_{0}$ is \spAdmissibleIn{\Theta}{\NSE{\Theta}}{\cC}.
\end{lemma}
\begin{proof}
Suppose $\Delta_0$ is not \spAdmissibleIn{\Theta}{\NSE{\Theta}}{\cC}.
Then, for some $\Delta \in \cC$ and $\theta_{0} \in \Theta$,
it holds that
\begin{gather}
\label{aoeuforall}
(\forall \theta\in \NSE{\Theta})(\sRisk(\theta,\Delta) \le \sRisk(\theta,\Delta_{0}))
\\
\text{ and }\quad
\label{aoeuexists}
\sRisk(\theta_{0},\Delta)  \not\approx \sRisk(\theta_{0},\Delta_{0}).
\end{gather}
From \cref{aoeuexists}, $\sRisk(\theta_{0},\Delta_{0})-\sRisk(\theta_{0},\Delta) > 2\gamma$ for some positive $\gamma\in \Reals$.
Let $A$ be the set of all $a\in \NSE{\PosReals}$ such that
\[
 (\forall t\in \NSE{\Theta})\ (\NSE{d}(t,\theta_0)<a \implies \sRisk(t,\Delta_{0})-\sRisk(t,\Delta) > \gamma).
\]
By the S-continuity of $\sRisk$ on $\NS{\NSE{\Theta}}$, the set $A$ contains all infinitesimals.
By \cref{spillover} and the fact that $A$ is an internal set,
$A$ must contain some positive $a_0\in \Reals$.
In summary,
\[
(\forall t\in \NSE{\Theta})\ (\NSE{d}(t,\theta_0)<a_0 \implies \sRisk(t,\Delta_{0})-\sRisk(t,\Delta) > \gamma).
\]

Let $M=\{t\in \NSE{\Theta}: \NSE{d}(t,\theta_0)<a_0\}$.
By the internal definition principle, $M$ is an internal set.
By \cref{aoeuforall} and the definition and internality of $M$,
the difference in internal Bayes risk between $\Delta_{0}$ and $\Delta$ satisfies
\[
\sRisk(\pi,\Delta_{0})-\sRisk(\pi,\Delta)
&=\int_{\NSE{\Theta}}(\sRisk(t,\Delta_{0})-\sRisk(t,\Delta)) \pi(\dee t)\\
&\ge \int_{M}(\sRisk(t,\Delta_{0})-\sRisk(t,\Delta)) \pi(\dee t)
> \gamma \, \pi(M).
\]

But
$\gamma \, \pi(M)>\epsilon$
because $\pi$ is $\epsilon$-regular,
hence $\Delta_{0}$ is not \seBayes{\epsilon}{\cC} with respect to $\pi$.
\end{proof}

The following theorem is an immediate consequence of \cref{ensblyth}
and is a nonstandard analogue of Blyth's Method \citep[][\S5 Thm.~7.13]{LC98} (see also \citep[][\S5 Thm.~8.7]{LC98}). In Blyth's method, a sequence of (potentially improper) priors with sufficient support is used to establish the admissibility of a decision procedure. In contrast, a single nonstandard prior witnesses the nonstandard admissibility of a nonstandard Bayes procedure.

\begin{theorem}\label{eextrpimplyea}
Suppose $\Theta$ is a metric space and \cref{assumptionrc} holds.
Let $\delta_{0} \in \FRRE$ and $\cC \subset \FRRE$.
If there exists $\epsilon\in \NSE{\PosReals}$
such that $\sdelta_0$ is \seBayes{\epsilon}{\Ext{\cC} = \{ \sdelta : \delta \in \cC \}}
with respect to an $\epsilon$-regular nonstandard prior $\pi$,
then $\sdelta_{0}$ is \spAdmissibleIn{\Theta}{\NSE{\Theta}}{\Ext{\cC}}.
\end{theorem}

\begin{proof}
By \cref{assumptionrc}
and \cref{contScont},
for all $\delta \in \FRRE$, $\theta_0\in \Theta$, and $t \approx \theta_0$, we have
$\sRisk(t,\sdelta)\approx \sRisk(\theta_{0},\sdelta)$.
By \cref{ensblyth}, $\sdelta_{0}$ is \spAdmissibleIn{\Theta}{\NSE{\Theta}}{\Ext{\cC}}.
\end{proof}

These theorems have the following consequence for standard decision procedures:

\begin{theorem}\label{eblythresult}
Suppose $\Theta$ is a metric space and \cref{assumptionrc} holds,
and let $\delta_0\in \FRRE$ and $\cC \subseteq \FRRE$.
If there exists $\epsilon\in \NSE{\PosReals}$
such that $\sdelta_0$ is \seBayes{\epsilon}{\Ext{\cC} = \{\sdelta : \delta \in \cC \}}
with respect to an $\epsilon$-regular nonstandard prior,
then $\delta_0$ is \Admissible{\cC}.
\end{theorem}
\begin{proof}
The result follows from \cref{admissibilitythm}
and
\cref{eextrpimplyea}.
\end{proof}

\cref{eblythresult} implies the well-known result that Bayes procedures with respect to priors with full support are admissible \citep[][\S2.3 Thm.~3]{Ferguson} (see also \citep[][\S5 Thm.~7.9]{LC98}).

\begin{theorem}\label{regadmissible}
Suppose $\Theta$ is a metric space and \cref{assumptionrc} holds and let $\delta_{0} \in \FRRE$.
If $\delta_0$ is \Bayes{\FRRE} with respect to a prior $\pi$ with full support,
then $\delta_0$ is \Admissible{\FRRE}.
\end{theorem}
\begin{proof}
Note that $\delta_0$ is \BayesP{\pi}{\FRRE} if and only if $\sdelta_0$ is \NSBayesP{\NSE{\pi}}{\ExtFRRE}.
As $\pi$ has full support, $\NSE{\pi}$ is $\epsilon$-regular for every infinitesimal $\epsilon \in \NSE{\PosReals}$.
By \cref{eblythresult}, we have the desired result.
\end{proof}

We close with an admissibility result requiring no additional regularity:

\begin{theorem}\label{eBayesfullsupport}
Let $\delta_{0} \in \FRRE$
and $\cC \subseteq \FRRE$.
If there exists $\epsilon \in \NSE{\PosReals}$ such that
 $\sdelta_{0}$ is \seBayes{\epsilon}{\NSE{\cC}}
with respect to a nonstandard prior $\pi$ satisfying $\pi\{\theta\} \gg \epsilon$ for all $\theta \in \Theta$,
then $\delta_{0}$ is \Admissible{\cC}.
\end{theorem}

\begin{proof}
Suppose $\delta_{0}$ is not \Admissible{\cC}.
Then by \cref{admissibilitythm},
$\sdelta_{0}$ is not \spAdmissibleIn{\Theta}{\NSE{\Theta}}{\Ext{\cC}}.
Thus there exists $\delta \in \cC$  and $\theta_{0} \in \Theta$ such that $\sRisk(\theta,\sdelta) \le \sRisk(\theta,\sdelta_{0})$ for all $\theta \in \NSE{\Theta}$ and $\sRisk(\theta_{0},\sdelta_{0}) - \sRisk(\theta_{0},\sdelta) > \gamma$ for some $\gamma \in \PosReals$.
Then
$
\sRisk(\pi,\sdelta_{0}) - \sRisk(\pi,\sdelta) \ge \pi\{\theta_{0}\} \gamma > \epsilon.
$
But this implies that $\sdelta_{0}$ is not \seBayesP{\epsilon}{\pi}{\cC}.
\end{proof}

\begin{remark}\label{remtopoad}
The astute reader may notice that \cref{eBayesfullsupport}
is actually a corollary of \cref{eblythresult} provided we adopt the discrete topology/metric on $\Theta$.
Changing the metric changes the set of available prior distributions and also changes the set of $\epsilon$-regular nonstandard priors.  See also \cref{remarktopo}.
\end{remark}

\section{Some Examples}\label{sec: examples}

The following examples serve to highlight some of the interesting properties of our nonstandard theory and its consequences for classical problems.

\begin{example}
Consider any standard statistical decision problem with a finite discrete (hence compact) parameter space.
\cref{assumptionrc} holds trivially, and so
\cref{stbayes,fsestronger} imply that a decision procedure is extended admissible if and only if it is extended Bayes if and only if it is Bayes if and only if its extension is nonstandard Bayes among all internal decision procedures.
By \cref{regadmissible}, we obtain another classical result: if a procedure is Bayes with respect to a prior with full support, it is admissible.
\end{example}

\begin{example}\label{normallocprob}
Consider the classical problem of estimating the mean of a multivariate normal distribution in $d$ dimensions under squared error when the covariance matrix is known to be the identity matrix.
By the convexity of the squared error loss function,
\cref{FREconvexECC} implies the nonrandomized procedures form an essentially complete class (indeed, the loss is strictly convex and so the nonrandomized procedures are actually a complete class).
\cref{fsecomplete} implies that every extended admissible estimator among $\FiniteRiskEstimators$ is nonstandard Bayes among $\Ext{\FRREFS}$.
We can derive further results by noting that risk functions are continuous, which follows from a general theorem on exponential families:

\begin{theorem}[{\citep[][\S5 Ex.~7.10]{LC98}}]
Assume $\Model$ is an exponential family.
Then, for any loss function $\Loss$ such that the risk is always finite,
the risk function is continuous.
\end{theorem}

Thus \cref{assumptionrc} holds.
\cref{regadmissible} then implies that every Bayes estimator with respect to a prior with full support is admissible.
In particular, for every $k > 0$, the estimator $\delta^{B}_{k}(\bx) = \frac {k^2}{k^2+1} \bx$ is Bayes
with respect to the full-support prior $\pi_{k} = \Normal{0}{k^2 \Identity_d}$, hence admissible.

Consider now the maximum likelihood estimator $\delta^{M}(\bx) = \bx$ and let $K$ be an infinite natural number.
Then $\sdelta^M(\bx) \approx (\sdelta^{B})_{K}(\bx)$ for all $\bx \in \NS{\HReals^d}$, where $\sdelta^{B}$ is the extension of the function $k \mapsto \delta^{B}_{k}$.
The normal prior $(\NSE{\pi})_{K}$ is ``flat'' on $\Reals$ in the sense that, at every near-standard real number, the density is within an infinitesimal of $(2 \pi)^{-\frac 1 2} K^{-d}$.
These observations provide a nonstandard interpretation to the idea that the MLE estimator is a Bayes estimator given a ``uniform'' prior.

Since \cref{assumptionrc} holds,
\cref{eblythresult}
implies that every estimator
whose extension is \seBayes{\epsilon}{\sFiniteRiskEstimators}
with respect to an $\epsilon$-regular prior is \Admissible{\FiniteRiskEstimators}.
An easy calculation reveals that the Bayes risk of $(\sdelta^{B})_{K}$ with respect to $(\NSE{\pi})_{K}$ is $d \frac {K^2}{K^2+1}$,
while the Bayes risk of $\sdelta^M$ with respect to $(\NSE{\pi})_{K}$ is $d$.  Thus, $\sdelta^M$ is even \NSBayes{\sFRRE},
and in particular, $\sdelta^M$ is \seBayes{\epsilon}{\sFRRE} for $\epsilon = (K^2+1)^{-1}$.
From the density above, it is then straightforward to verify that, for $d=1$ and $d=2$, the prior $(\NSE{\pi})_{K}$ is $\epsilon$-regular, but that it fails to be for $d \ge 3$. Therefore, by \cref{eblythresult}, it follows that $\delta^M$ is \Admissible{\FiniteRiskEstimators}
for $d=1$ and $d=2$, as is well known. The theorem is silent in this case for $d \ge 3$.
Indeed, \citet{stein54} famously showed that $\delta^M$ is inadmissible for $d \ge 3$.
\end{example}

\begin{remark}\label{remarktopo}
Here we have used \cref{eblythresult} and the standard metric on $\Theta=\Reals^d$ in order to establish admissibility.  Note that the infinite-variance Gaussian prior is not $\epsilon$-regular with respect to the discrete metric on $\Theta$,
and so a different nonstandard prior would have been needed to establish admissibility via \cref{eBayesfullsupport}.
\end{remark}

In \cref{sec: CCs under compactness and RC},
we established that class of Bayes procedures coincides with the class of extended admissible estimators
under compactness of the parameter space and continuity of the risk.
The next example demonstrates that extended admissibility and Bayes optimality do not necessarily align
if we drop the risk continuity assumption, even when the parameter space is compact. We study a non-Bayes admissible estimator and characterize a nonstandard prior with respect to which it is nonstandard Bayes.

\begin{example} \label{bernoulliexample}
Let $X = \{0,1\}$ and $\Theta = [0,1]$, the latter viewed as a subset of Euclidean space.
Define $g : [0,1] \to [0,1]$ by $g(x) = x$ for $x > 0$ and $g(0) = 1$,
and let $P_t = \Bernoulli(g(t))$, for $t \in [0,1]$,
where $\Bernoulli(p)$ denotes the distribution on $\{0,1\}$ with mean $p \in [0,1]$.
Every nonrandomized decision procedure $\delta : \{0,1\} \to [0,1]$ thus corresponds
with a pair $(\delta(0),\delta(1)) \in [0,1]^2$, and so we will express
nonrandomized decision procedures as pairs.
Consider the loss function $\ell(x,y) = (g(x) - y)^2$.
(For every $x$, the map $y \mapsto \ell(x,y)$ is convex but merely lower semicontinuous on $[0,1]$.
It follows from \cref{FREconvexECC} that nonrandomized procedures form an essentially complete class.)
\end{example}

\begin{theorem}
In \cref{bernoulliexample}, $(0,0)$ is an admissible non-Bayes estimator.
\end{theorem}

\begin{proof}
Let $(a,b) \in [0,1]^2$ and let $c = \min \{a,b\}$.
For every $n \in \Nats$, we have
\[
r(n^{-1},(a,b)) = (1-n^{-1}) \Loss(1/n,a) + n^{-1}\, \Loss(1/n,b)
\]
and so, for sufficiently large $n$, we have
$r(n^{-1},(a,b)) \ge r(n^{-1},(c,c))$.
But, for every $d > 0$ and sufficiently large $n$, it also holds that $r(n^{-1},(d,d)) > r(n^{-1},(0,0))$.
Hence, $(a,b)$ does not dominate $(0,0)$, hence $(0,0)$ is admissible.

To see that $(0,0)$ is not Bayes,
note that an estimator $(a,b)$ has the same Bayes risk under $\pi$ as it would
under the (pushforward) prior $\nu = \pi \circ g^{-1}$ in the
statistical decision problem with sample space ${X}$, parameter space $\Theta' = g(\Theta) = (0,1]$,
model $P'_t = \Bernoulli(t)$, and squared error loss $\ell'(x,y) = (x,y)$. However, in this case, the loss is strictly convex and so
 the Bayes optimal decision is unique and is the posterior mean, which is a value in $(0,1]$, hence $(0,0)$ cannot be Bayes optimal for any prior.
\end{proof}

The failure of $(0,0)$ to be Bayes optimal is due to the fact that the posterior mean cannot be $0$.
However, in the nonstandard universe, the posterior mean can be made to be infinitesimal, in which case the Bayes risk of $(0,0)$ is also infinitesimal.

\begin{theorem}
$\NSE{(0,0)}$ is nonstandard Bayes with respect to any prior concentrating on some infinitesimal $\epsilon > 0$.
\end{theorem}
\begin{proof}
Pick any positive infinitesimal $\epsilon$
and consider the nonstandard prior $\pi$ concentrated on $\epsilon$.
The nonstandard Bayes risk of $(0,0)$ with respect to $\pi$ is
\[
\sRisk(\pi,(0,0))=\sRisk(\epsilon,(0,0))=\epsilon(\epsilon-0)^2+(1-\epsilon)(\epsilon-0)^2\approx 0
\]
Because the loss function in \cref{bernoulliexample} is nonnegative, $(0,0)$ must be a nonstandard Bayes estimator with respect to $\pi$.
\end{proof}

We close by observing that $(0,0)$ is a generalized Bayesian estimator.
In particular, the generalized Bayes risk with respect to the improper prior $\pi(\dee \theta) = \theta^{-2} \dee \theta$ is finite, whereas every other estimator has infinite Bayes risk.
The modified statistical decision problem with parameter space $\Theta' = (0,1]$ under the standard topology, model $\Model'$ and loss $\Loss'$ meets the hypotheses of \cref{berger}---
indeed, the modified problem is that of estimating the mean of an exponential family model---
hence every extended admissible procedure is generalized Bayes.
The original problem does not meet the hypotheses of \cref{berger}, since the loss is not jointly continuous.

\section{Miscellaneous remarks}
\label{sec: remarks}

\begin{enumerate}[label=(\roman*),wide, labelwidth=!, labelindent=1em]

\item \label{topodiscussion}
We have required $\Theta$ to be Hausdorff in order for the standard part map to be uniquely defined. Relaxing this assumption would require that we work with a standard part relation instead. At this moment, we see no roadblocks.

\item
Assume $\sdelta$ is \NSBayes{\Ext{\cC}}. Under what conditions can we conclude that $\sdelta$ is \NSBayes{\NSE{\cC}}? \seBayes{0}{\Ext{\cC}}? Among $\NSE{\cC}$?
In \cref{fsestronger}, we show that these conclusions follow for $\cC = \FRREFS$ when $\Theta$ is compact Hausdorff and \cref{assumptionrc} holds.
Can we weaken these conditions or find incomparable ones?
A related problem is to identify conditions under which $\delta$ is \spExtAdmissible{\Theta}{\NSE{\cC}}.
As a starting point, it is an open problem to find a procedure $\delta$ such that either
1) $\sdelta$ is \NSBayes{\sFRREFS} but $\sdelta$ is not \seBayes{0}{\Ext{\FRREFS}}, or
2) $\sdelta$ is \NSBayes{\Ext{\FRREFS}} but $\sdelta$ is not \NSBayes{\sFRREFS}.
Note that \cref{normallocprob} demonstrates that $\delta^{M}$ is \NSBayes{\sFRREFS} but
not \seBayes{0}{\sFRREFS}.

\item
We restricted our attention to decision procedures whose risk functions are everywhere finite.
However, if we do not make this restriction, it is possible for an admissible decision procedures to have infinite risk in some state $\theta \in \Theta$  \citep[][\S4A.13 Part (iv)]{Brown86}.
We make repeated use of the finite risk property and so it would be an interesting contribution to relax this assumption.
A related issue is our restriction to nonnegative real-valued loss functions. It would be straightforward to allow loss functions that are bounded below or above. Allowing arbitrary loss functions, however, raises the possibility that a decision procedure's risk could be undefined on some subset of the parameter space.

\item
It is worth searching for a converse to \cref{eBayesfullsupport}, perhaps with a view to identifying a nonstandard analogue of Stein's necessary and sufficient condition for admissibility \citep{MR0070929}, but one witnessed by a single (nonstandard) prior distribution.

\item
Our standard result,  \cref{stbayes},
is similar to \cref{bergercompact}
of \citet{Berger85} and \cref{waldcompact} of \citet{Wald47}.
Our theorem identifies the class of extended admissible procedures and the class of Bayes procedures,
and does so by assuming that risk functions are continuous and the parameter space is compact Hausdorff.
These are weaker assumptions than those of Berger, and more natural than those of Wald.
It would take some work to understand which assumptions of theirs are needed to show that the extended admissible procedures (equivalently, the Bayes procedures) form a complete class. In our opinion, it is preferable to understand conditions under which we can
identifying extended admissibility and Bayes optimality and then separately understand conditions under which the former is a complete class. (The classical textbook by Blackwell and Girschick \citep{BG54} adopts a similar aesthetic principle.)
We believe that the methods developed in this paper may allow us to remove or generalize regularity conditions in other existing results.

\item
It would be illuminating to uncover a complete characterization of the relationships between nonstandard Bayes procedures, extended Bayes procedures, limits of Bayes procedures, and generalized Bayes procedures. Some connections can be identified simply by transfer: e.g., we already know that extended Bayes procedures are nonstandard Bayes by a simple transfer argument.
Given our theorems connecting extended admissibility and nonstandard Bayes optimality, progress on this question immediately yields new connections between extended admissibility and these relaxed notions of Bayes optimality.

\item

Under compactness and risk continuity, extended admissible procedures are nonstandard Bayes among all internal decision procedures.  In general, however, an extended admissible procedure is nonstandard Bayes only among the nonstandard extensions of standard procedures.  Can we find an example witnessing the gap between these two results? In particular, can we identify a decision problem
for which the extension of some (all, most) (extended) admissible decision procedure(s) is (are) not nonstandard Bayes among all internal decision procedures?

\end{enumerate}

\section*{Acknowledgments}

The authors owe a debt of gratitude to
William Weiss
for detailed suggestions, as well as for his assistance with set theoretic and topological issues.
We  thank
Gintar\.e D\v{z}iugait\.e,
Cameron Freer, and
H.~Jerome Keisler
for early discussions and insights,
and thank
Nate Ackerman, 
Michael Evans,
Arno Pauly,
and
Aaron Smith
for feedback on drafts
and helpful discussions.  Finally, the authors
would like to thank Peter Hoff for his course notes, which served as our first introduction to the topic.
Work on this publication was made possible through 
an NSERC Discovery Grant, Connaught Award, and U.S. Air Force Office of
Scientific Research grant \#FA9550-15-1-0074.
This work was done in part while DMR was visiting the Simons Institute for the Theory of Computing at UC Berkeley.

\makebib

\appendix

\section{Nonstandard analysis --- Basic Notions and Key Results}\label{intronsa}

In this appendix, we give a brief introduction for those readers not familiar with nonstandard analysis.
The interested reader
can find a thorough introduction in \citep{AR65}. For modern applications of nonstandard analysis, one can read \citep{NSAA97} or \citep{NDV}.
Our following introduction of nonstandard analysis owes much to \citep{NSAA97}.

For a set $S$, let $\PowerSet{S}$ denote its power set.
Given any set $S$, define $\bV_{0}(S)=S$ and $\bV_{n+1}(S)=\bV_{n}(S)\cup\PowerSet{\bV_{n}(S)}$ for all $n\in \Naturals$. Then $\bV(S)=\bigcup_{n\in \Naturals}\bV_{n}(S)$ is called the \defn{superstructure} of $S$,
and $S$ is called the \defn{ground set} of the superstructure $\bV(S)$. We treat the members of $S$ as indivisible atomics.
The \defn{rank} of an object $a\in \bV(S)$ is the smallest $k$ for which $a\in \bV_{k}(S)$.
The members of $S$ have rank 0. The objects of rank no less than $1$ in $\bV(S)$ are precisely the sets in $\bV(S)$. The empty set $\emptyset$ and $S$ both have rank 1.

We now formally define the language $\Language(\bV(S))$:
For every element $c \in \bV(S)$, we introduce a constant symbol $\bar c$.
We also fix a countable collection of variable symbols $v_1, v_2, \dots$, distinct from the constants and the following reserved symbols: $=$, $\in$, $)$, $($, $\land$, $\lor$, $\neg$, $\forall$, and $\exists$.
The (bounded) \defn{formulas} in $\Language(\bV(S))$ are defined recursively:

\begin{itemize}
\item If $u$ and $v$ are constant/variable symbols, then $(u=v)$ and $(u \in v)$ are formulas.

\item If $\phi$ and $\psi$ are formulas, then $(\phi \land \psi),(\phi\lor \psi)$, and $(\neg\phi)$ are formulas.

\item If $\phi$ is a formula and $u$ is a variable, and $v$ is a constant or variable symbol distinct from $u$, 
then $(\forall u \in v)(\phi)$ and $(\exists u \in v)(\phi)$ are formulas.

\end{itemize}

A variable $x$ is \defn{free} in a formula $\phi$ if it is not within the scope of any quantifiers. 
More carefully:
$x$ is free in $(u=v)$ if either $u$ or $v$ is the variable $x$, and similarly for $(u \in v)$;
$x$ is free in $(\phi \land \psi)$ if it is free in either $\phi$ or $\psi$, and similarly for $\lor$ and $\neg$;
$x$ is free in $(\forall u \in v)(\phi)$ if $x$ is not $u$ and $x$ is $v$ or $x$ is free in $\phi$, and similarly for $(\forall u \in v)(\phi)$.  We will sometimes write $\phi(x_1,\dotsc,x_n)$ to mean that $x_1,\dotsc,x_n$ are exactly the free variables in $\phi$, and write $\phi(c_1,\dots,c_n)$ to mean the formula $\phi$ where we substitute every free occurrence of $x_i$ with the constant $c_i$, for every $i=1,\dots,n$. 
(See \citep[][]{MR1059055} for a careful description of substitution.)

A sentence is a formula with no free variables.  
Informally, sentences in $\Language(\bV(S))$ are either true or false statements about $\bV(S)$.  We formalize this here via statements of the form $\phi$ \defn{holds in} $\bV(S)$.
\begin{definition} 
Let $\phi,\psi$ be sentences in $\Language(\bV(S))$. 
\begin{itemize}
\item $(\bar c_1 = \bar c_2)$ holds in $\bV(S)$ if and only if $c_1 = c_2$;
\item $(\bar c_1 \in \bar c_2)$ holds in $\bV(S)$ if and only if $c_2$ has rank 1 or higher and $c_1\in c_2$;
\item $(\phi \land \psi)$ holds in $\bV(S)$ if and only if $\phi$ holds in $\bV(S)$ and $\psi$ holds in $\bV(S)$, and similarly for $\lor$ and $\neg$;
\item $(\forall u \in \bar c)(\phi(u))$ holds in $\bV(S)$ if and only if,
        for all $x \in c$, the sentence $\phi(c)$ holds in $\bV(S)$, and similarly for $(\exists u \in v)(\phi(u))$.
\end{itemize}
\end{definition}

We will also write that $\phi$ is true in $\bV(S)$ to mean $\phi$ holds in $\bV(S)$.  When the structure $\bV(S)$ is clear from context, we will simply write that $\phi$ holds or that $\phi$ is true.
We will use the following abbreviations: $(\phi\implies \psi)$ for $((\neg \phi)\lor (\psi))$ and $(\phi \Longleftrightarrow \psi)$ for $(\phi\implies \psi)\land (\psi\implies \phi)$.

It may seem necessary (or at least useful) to include function and relation symbols in our language for every function and relation in $C \subseteq \bV(S)$.
For example, one would expect $(\exists x \in \bar\Reals)\, (\bar{1} \bar{+} \bar{x} = \bar{3})$ 
to be a well-formed sentence in $\Language(\bV(\Reals))$. (Writing $\bar{}$ over every constant is cumbersome and unnecessary in practice, and so we will drop this symbol from now on.)
In fact, every relation and function symbol is, in effect, already available, because
there is a mechanical way to translate formulas with function and relation symbols into the above language (and in a way that commutes with taking ${}^{*}$-transfers as defined below).  Informally, the translation works as follows: First, function symbols are removed by rewriting them in terms of their graph relations using additional quantifiers.  Second, because every relation is an element in $\bV(S)$ and thus a constant in the language, relation symbols can be replaced with statements of the form $a \in b$, where $a$ is a tuple.  Finally, tuples can be mechanically encoded once elements of products spaces are encoded as sets of sets in a canonical way.
Thus our language is powerful enough to represent formula with relation and function symbols, and so these symbols will be used without comment.

Let $S$ and $\NSE{S}$ be a pair of ground sets,
and let $\NSE{} \colon \bV(S) \to \bV(\NSE{S})$ be a map between their superstructures that preserves rank. Relative to this map, an element $a \in \bV(\NSE{S})$ is \defn{internal} when
there exists $b\in \bV(S)$ such that $a\in \NSE{b}$, and $a$ is said to be \defn{external} otherwise. The language of $\bV(\NSE{S})$ is almost the same as $\Language$ except that we enlarge the set of constants to include every element in $\bV(\NSE{S})$. We denote the language of $\bV(\NSE(S))$ by $\Language(\bV(\NSE{S}))$.
If $\phi(x_1,\dotsc,x_n)$ is a formula in $\Language(\bV(S))$ with free variables $x_1,\dotsc,x_n$, then the $*$-transfer of $\phi$ is the formula in $\Language(\bV(\NSE{S}))$ obtained by changing every constant $a$ to $^*a$. Clearly, every constant in $\NSE{\phi(x_1,\dotsc,x_n)}$ is internal.

Let $\kappa$ be an uncountable cardinal number.
A \defn{$\kappa$-saturated nonstandard extension} of a superstructure $\bV(S)$ is a set $\NSE{S}$ and a map
${}^* \colon \bV(S) \to \bV(\NSE{S})$ satisfying: %
\begin{itemize}

\item \emph{extension}:
$\NSE{S}$ is a superset of $S$ and $\NSE{s}=s$ for all $s\in S$.

\item \emph{transfer}:
For every sentence $\phi$ in $\Language(\bV(S))$,
$\phi$ holds in $\bV(S)$ if and only if its $*$-transfer ${\NSE{\phi}}$ holds in $\bV(\NSE{S})$.
\item \emph{$\kappa$-saturation}:
For every family $\FSC=\{A_{i}:i\in I\}$ of internal sets
indexed by a set $I$ of cardinality less than $\kappa$,
if $\FSC$ has the finite intersection property, i.e., if every finite intersection of elements in $\FSC$ is nonempty,
then the total intersection of $\FSC$ is nonempty.

\end{itemize}

A $\aleph_{1}$-saturated model can be constructed via an ultrafilter, see \citep[][Thm.~1.7.13]{NSAA97}. However, saturation to any uncountable cardinal number is possible:

\begin{theorem}[{\citep{luxemberger}}]
For every superstructure $\bV(S)$ and uncountable cardinal number $\kappa$,
there exists a $\kappa$-saturated nonstandard extension of $\bV(S)$.
\end{theorem}

From this point on, we shall always assume that our nonstandard extension is at least $\aleph_1$ saturated.

It is easy to see that every element of $\NSE{S}$ is internal. It is also clear that $A\subset \NSE{S}$ is internal if and only if $A\in \NSE{\PowerSet{S}}$. Most sets are external subsets of $\NSE{S}$, and external subsets play an important role.
The main tool for constructing internal sets is the internal definition principle:

\begin{lemma}[Internal Definition Principle]
Let $\phi(x)$ be a formula in the language $\Language(\bV(\NSE{S}))$ with free variable $x$. Suppose that all constants
that occurs in $\phi$ are internal, then $\{x \in \bV(\NSE{S}): \text{$\phi(x)$ holds in $\bV(\NSE{S})$} \}$ is internal in $\bV(\NSE{S})$.
\end{lemma}

Saturation can be equivalently expressed in terms of the satisfiability of families of formulas.
The role of the finite intersection property is played by finite satisfiability:

\begin{definition}
Let $J$ be an index set and let $A \subseteq \bV(\NSE S)$.
A set of formulas $\{\phi_{j}(x) \mid j \in J\}$ in $\Language(\bV(\NSE S))$ is said to be
\defn{finitely satisfiable in $A$} when, for every finite subset $\alpha \subset J$,
there exists $c\in A$ such that $\phi_{j}(c)$ holds in $\bV(\NSE S)$ for all $j\in \alpha$.
\end{definition}

We provide the following alternative expression of $\kappa$-saturation:

\begin{theorem}[{\citep[][Thm.~1.7.2]{NSAA97}}]
Let $\bV(\NSE{S})$ be a $\kappa$-saturated nonstandard extension of the superstructure $\bV(S)$, where $\kappa$ is an uncountable cardinal number.
Let $J$ be an index set of cardinality less than $\kappa$.
Let $A$ be an internal set. %
For each $j\in J$, let $\phi_j(x)$ be a formula in $\Language(\bV(\NSE{S}))$
whose constants are internal objects.
Further, suppose that the set of formulas $\{\phi_j(x) \mid j\in J\}$ is finitely satisfiable in $A$.  Then there exists $c\in A$ such that $\phi_j(c)$ holds in $\NSE{\bV}(S)$ simultaneously for all $j\in J$.
\end{theorem}

\subsection{The hyperreals}\label{sec:hyreal}

Take $S=\Reals$.  Then any nonstandard extension ${\HReals}$ is a superset of $\Reals$ and $\NSE{x}=x$ for $x\in \Reals$.
For every $n \in \Nats$, let $A_n = \{ x \in \Reals: 0 < x < \frac 1 n \}$.  Then $\NSE{\!A_n}$ is internal by definition and, by the transfer principle, we have that $\NSE{\!A_n} = \{ x \in \NSE{\Reals} \st 0 \;{}^*\!\!\!< x \;{}^*\!\!\!< \frac 1 n \}$, which is clearly seen to be internal by the internal definition principle.  (For standard functions---like addition and multiplication---and relations---like the less-than relation---we will drop the $^*$ when the context is clear.)

Let $\FSC=\{\NSE{\!A_{n}}:n\in\Naturals\}$.
Clearly, $\FSC$ has the finite intersection property, which can be seen to hold by transfer, because each $A_n$ is nonempty and $A_n \supseteq A_{n+1}$.
By the saturation principle, the total intersection of $\FSC$ is nonempty.
Therefore, there exists $x\in {\HReals}$ such that $0<x<\frac{1}{n}$ for all $n\in \Naturals$.
Such a number is called an \defn{infinitesimal}.
It follows that there exists $y \in \NSE{\Reals}$ such that $y > n$ for all $n \in \Naturals$.
Such a number is called an \defn{infinite number}.
It follows that ${\HReals}$ is a proper superset of $\Reals$.

Write $x\approx y$ if $|x-y|$ is infinitesimal
and write $\mu(x) = \{y \in \NSE{\Reals}: x \approx y \}$ for the \defn{monad of $x$}.
An element $x\in {\HReals}$ is \defn{near-standard} if $x \approx a$ for some $a\in \Reals$.
An element $x\in {\HReals}$ is \defn{finite} if $|x|$ is bounded by some standard real number $a$.
By the compactness of closed intervals of $\Reals$ and the transfer principle,
one can show that
an element $x\in {\HReals}$ is finite if and only if $x$ is near-standard.

It is also straightforward to show that, if $x\in {\HReals}$ is near-standard, then there exists a unique real $a\in \Reals$, called the \defn{standard part of $x$}, such that $x\approx a$.
Let $\NS{\NSE{\Reals}}$ denote the collection of all near-standard points in ${\HReals}$,
and let $\ST\colon \NS{\NSE{\Reals}} \to \Reals$ denote the \defn{standard part map} taking a near-standard point $x$ to its standard part.
For a near-standard real $x$, we will write $\SP{x}$ instead of $\ST(x)$;
for a function $f$ into $\NS{\NSE{\Reals}}$, we will write $\SP{f}$ to denote the composition $x \mapsto \ST (f(x))$;
for a set $A \subseteq \NS{\NSE{\Reals}}$, we will write $\ST(A)$ to denote the set $\{x\in {\Reals}: (\exists a\in A) \, \ST(a) = x \}$;
and, for a set $B \subseteq \Reals$, we will write $\ST^{-1}(B)$ to denote the inverse image $\{ x\in {\NSE{\Reals}}: (\exists b\in B) \, x \approx b \}$.
For a standard real $x$, the set $\ST^{-1}(x)$ is the set of nonstandard numbers infinitesimally close to $x$, i.e., it is the monad $\mu(x)$ of $x$.

The following example is a classical example of external set:

\begin{example}
Let $\phi$ be the sentence: $\forall A\in \PowerSet{\Reals}$, if $A$ is bounded above, then $A$ has a least upper bound.\footnote{Formally, $\phi$ is the sentence
$
(\forall A\in \PowerSet{\Reals})
(( (\exists a \in \Reals) (\forall x \in A) (x < a))
  \implies
  ((\exists a \in \Reals)
     (((\forall x \in A) (x < a)) \land ((\forall b \in \Reals) ( ((\forall x \in A) (x < a)) \implies (a \le b) ) ))
     )
  )$, but we will prefer to work informally for clarity, although caution must be exercised:
the ``transfer'' of the sentence
``$\forall A\subset \Reals$ if $A$ is bounded above then $A$ has a least upper bound''
might lead one to conclude that all bounded sets have least upper bounds.
The error here is that $\subset$ is not in the first order language of set theory, and so the sentence is malformed.
}
By the transfer principle, it holds that, for all $A\in \NSE{\PowerSet{\Reals}}$, i.e., all internal subsets of ${\HReals}$,
if $A$ is bounded above, then $A$ has a least upper bound.
Suppose the monad $\mu(0)$ is internal.
Then there exists
$a_{0}\in \NSE{\Reals}$ such that $a_{0}$ is a least upper bound for $\mu(0)$.
Clearly $a_{0}>0$. Suppose $a_{0} \in \mu(0)$. Then $2a_{0} \in \mu(0)$ and $2a_{0}>a_{0}$, contradicting that $a_0$ is an upper bound, hence $a_0 \not\in\mu(0)$.
But then $\frac{a_{0}}{2} \not\in \mu(0)$ and $\frac{a_{0}}{2} < a_0$, contradicting that $a_0$ is the least upper bound. Hence, $\mu(0)$ is external.
\end{example}

The following lemma collects together two simple facts we use repeatedly:
\begin{lemma}\label{standardnums}
Let $a,b \in \NS{\HReals}$.
\begin{enumerate}
\item $a \approx b$ if and only if $\SP{a} = \SP{b}$.
\item $a \lessapprox b$ implies $\SP{a} \le \SP{b}$.
\end{enumerate}
\end{lemma}

We conclude this section by introducing two useful properties derived from saturation:

\begin{theorem} [{\citep[][Prop.~2.8.2]{NSAA97}}]\label{spillprinciple}
Let $A\subset \NSE{\Reals}$ be an internal set.
\begin{enumerate}
\item (Overspill) If $A$ contains arbitrarily large finite positive numbers, then it also contains an infinite number.

\item (Underspill) If $A$ contains arbitrarily small positive infinite numbers, then it also contains a finite number.
\end{enumerate}
\end{theorem}

An immediate consequence of this theorem is the following:

\begin{corollary} \label{spillover}
Let $A\subset \NSE{\Reals}$ be an internal set.
\begin{enumerate}
\item (Overspill) If $A$ contains arbitrarily large infinitesimals, then $A$ contains some non-infinitesimal element from $\NSE{\Reals}$.
\item (Underspill) If $A$ contains arbitrarily small positive non-infinitesimals, then $A$ contains some positive infinitesimal element from $\NSE{\Reals}$.
\end{enumerate}
\end{corollary}

\subsection{Nonstandard extensions of more general spaces}\label{sec:hytopo}

The concepts from the previous section can be generalized to a general metric space $\gX$ by replacing the standard metric in $\Reals$ by the corresponding metric in the metric space.
Indeed, all the concepts can be defined for an arbitrary topological space $(\gX,T)$.
In this case, we will assume that our model is more saturated than the cardinality of $T$.
The monadic structure is determined by $T$ (although we will elide the topology in our notation as
it will be clear from context).  In particular, for $x \in \NSE{\gX}$, the monad of $x$ is
defined by $\mu(x) = \bigcap \{ \NSE{G} : x \in G \in T \}$.
In general $\mu(x)$ is external, however,
for every $x \in \NSE{\gX}$, one can use saturation to prove that there exists an (internal) $\NSE{}$open set $B \in \NSE{T}$
such that $ x \in B \subseteq \mu(x)$.
Given a set $\gY \subseteq \NSE{\gX}$,
the near-standard points are given by $\NS{\gY} = \{ y \in \gY : (\exists x \in \gX)\, y \in \mu(x) \}$.
In general, $\NS{\NSE{\gX}}$ is a proper subset of $\NSE{\gX}$. However, when $\gX$ is compact, we have $\NS{\NSE{\gX}}=\NSE{\gX}$. Indeed, this is the nonstandard way to characterize a compact space.

\begin{theorem}[{\citep[][Thm.~3.5.1]{NSAA97}}]\label{nscompact}
Let $\gX$ be a Hausdorff space. A set $A\subset \gX$ is compact if and only if $\,\NSE{A}=\NS{\NSE{A}}$.
\end{theorem}

It is worth presenting an example:
\begin{example}
Consider $\NSE{[0,1]}=\{x\in {\HReals}: 0\leq x\leq 1\}$ under its standard topology.
The closed interval $[0,1]$ is compact, and so ${\NSE{[0,1]}}=\mathrm{NS}(\NSE{[0,1]})$.
However, the open interval $(0,1)$ is not compact, hence $\NSE{(0,1)}\neq \NS{\NSE{(0,1)}}$.
Indeed, consider a positive infinitesimal $\epsilon$. Then $\epsilon\in \NSE{(0,1)}$ but $\epsilon\not\in \NS{\NSE{(0,1)}}$.
If we adopt the discrete topology on $[0,1]$, then $\mu(x) = \{x\}$ for every $x \in \NSE{[0,1]}$
and the set of near-standard points $\NS{\NSE{[0,1]}}$ is precisely the set of standard points $[0,1]$ itself.
\end{example}

Nonstandard analysis gives a very succinct characterization of continuity in terms of $S$-continuity (\cref{Scontinuous}).
Here we give a proof of the classical result relating continuity to S-continuity.  This proof is generalized from the proof of \citep[][Thm.~2.4.1]{NSAA97}.

\begin{proof}[Proof of \cref{contScont}]
Let $x_1,x_2$ be two near-standard points in $\NSE{D}$ such that $x_1\approx x_2$.  Let $x_0=\ST(x_1)=\ST(x_2)$.
Then $x_0\in D$.
By local compactness, there exists a compact neighborhood $K_0$ of $f(x_0)$. Let $U_0 \subseteq K_0$ be an open set containing the point $f(x_0)$.
It is clear that $x_1,x_2$ are elements of $\NSE{(f^{-1}(U_0))}$ hence $\NSE{f}(x_1)$ and $\NSE{f}(x_2)$ are elements of $\NSE{K_0}$ by transfer. As $K_0$ is compact, both $\NSE{f}(x_1)$ and $\NSE{f}(x_2)$ are near-standard.

Suppose $\NSE{f}(x_1)\not\approx\NSE{f}(x_2) $.  For every $y\in \gY$, let $\nu_{\gY}(y)$ denote the monad of $y$. Then $\NSE{f}(x_1)\in \nu_{\gY}(y_1)$ and $\NSE{f}(x_2)\in \nu_{\gY}(y_2)$ for distinct $y_1,y_2\in f(D)$.
Because $\gY$ is Hausdorff,
there exists an open set $U_1$ containing $y_1$ but not $y_2$. Then $\NSE{f}(x_2)\not\in \NSE{U_1}\cap \NSE{f(D)}$ hence $x_2\not\in \NSE{(f^{-1}(U_1))}\cap \NSE{D}$. However, $x_1\in \NSE{(f^{-1}(U_1))}\cap \NSE{D}$ and $f^{-1}(U_1)\cap D$ is open in $D$. This shows that $x_1\not\approx x_2$, a contradiction.
\end{proof}

\subsubsection{Vector spaces}
\label{vectorspacessec}

We conclude this section with a short discussion of vector spaces over the nonstandard field $\HReals$.
The notions of convexity of sets and functions have their ordinary abstract semantics:
If $X$ is a vector space over the field $\HReals$ (where addition and multiplication are the nonstandard extensions of ordinary addition and multiplication),
a (possibly external) subset $A$ of  $X$ is convex if for all $a\in \NSE{[0,1]}$ and $x_1,x_2\in A$,  we have
$ax_1+(1-a)x_2\in A$.
A (possibly external) function $f$ from $\gX$ to ${\HReals}$ is \defn{convex} if its graph is a convex set, i.e., for all $a\in \NSE{[0,1]}$ and $x_1,x_2\in \gX$, $f(ax_1+(1-a)x_2)\leq af(x_1)+(1-a)f(x_2)$. The function is \defn{strictly convex} if the inequality is strict.
We will make use of several properties that hold when $X$ is an internal set, i.e., when $X$ is an \defn{internal vector space} (over the field $\HReals$).
By transfer, the first order characterization of internal vector spaces over $\NSE{\Reals}$ is the same as that of standard vector spaces over $\Reals$.
An important class of internal convex spaces are the hyperfinite-dimensional Euclidean spaces $\IFuncs{\HReals}{N}$  for $N \in \NSE{\Nats}$.

\begin{theorem}
$\IFuncs{\HReals}{N}$ is an internal convex space for every $N\in \NSE{\Nats}$.
\end{theorem}
\begin{proof}
Fix any $N\in \NSE{\Nats}$.
For every $n \in \Nats$, let $F_n$ denote the set of functions from $\{1,\dotsc,n\}$ to $\Reals$.
Then $\IFuncs{\HReals}{N}\in \NSE{\bigcup_{n\in \Nats}F_n}$ hence $\IFuncs{\HReals}{N}$ is internal.
It is straightforward to demonstrate convexity.
\end{proof}

We close with a proof of the hyperfinite separating hyperplane theorem, which demonstrates the use of transfer.

\begin{proof}[Proof of~\cref{hyperplanethm}]
We first restate the standard hyperplane separation theorem.
We shall view the set $\Reals^{\Naturals}$ as the set of functions from $\Naturals$ to $\Reals$.
For every element $x\in \Reals^{\Naturals}$, we use $x(k)$ to denote the value of the $k$-th coordinate of $x$ for any $k\in \Naturals$.
The standard hyperplane separation theorem is equivalent to:
\begin{quote}
For any two disjoint convex $S_{1},S_{2}\in \PowerSet{\Reals^{\Naturals}}$,
if $\exists k\in \Naturals$ such that
$\forall s\in S_{1}\cup S_{2}$
$\forall k'>k$
we have $s(k')=0$ then
$\exists a\in \Reals^{\Naturals} \setminus \{\vzero\}$
with $a(k')=0$ for all $k'>k$ such that
$\forall p_{1}\in S_{1},p_{2}\in S_{2}$
$((\forall k'>k, a(k')=0) \land (\IP{a}{p_{1}} \leq \IP{a}{p_{2}} ))$.
\end{quote}
By the transfer principle, we know that $\NSE{(\Reals^{\Nats})}$ denotes the set of all internal functions
from $\NSE{\Nats}$ to $\HReals$.
We shall view the inner product $\IP{\argdot}{\argdot}$ to be a function from
$\Reals^{\Naturals} \times \Reals^{\Naturals}$ to $\Reals$.
Note that $\forall p,s\in \Reals^{\Naturals}$
if $\exists k\in \Naturals$ such that $\forall k'>k$ we have
$s(k')=0$ then $\IP{p}{s}=\sum_{i=1}^{k}p(i)s(i)$.
Thus the nonstandard extension of $\IP{\argdot}{\argdot}$ is a function from
$\NSE{(\Reals^{\Naturals})} \times \NSE{(\Reals^{\Naturals})}$ to $\HReals$ satisfying the same property.

Now by the transfer principle we know that:
\begin{quote}
For any two disjoint convex sets $S_{1},S_{2}\in \NSE{\PowerSet{\Reals^{\Naturals}}}$.
If $\exists K\in \NSE{\Naturals}$ such that $\forall s\in S_{1}\cup S_{2}$ $\forall K'>K$ we have $s(K')=0$ then $\exists W\in \NSE{(\Reals^{\Naturals})} \setminus \{\vzero\}$ such that for all $ p_{1}\in S_{1},p_{2}\in S_{2}$ we have $((\forall K'>K, W(K')=0)\land\sum_{i=1}^{K}W(i)p_{1}(i)\leq \sum_{i=1}^{K}W(i)p_{2}(i))$.
\end{quote}

In this sentence, it is easy to see that we can view the projections of $S_{1},S_{2}$ as internal subsets of $\IFuncs{\HReals}{K}$ and the projection of $W$ as an element from $\IFuncs{\HReals}{K} \setminus \{\vzero\}$.
Hence we have that:
$\forall K\in \NSE{\Naturals}$,
if $S_{1},S_{2}$ are two disjoint internal convex subsets of $\IFuncs{\HReals}{K}$,
then there exists $W\in \IFuncs{\HReals}{K} \setminus \{\vzero\}$ such that
for any $P_1\in S_{1}$ and any $P_{2}\in S_{2}$,
$\sum_{i=1}^{K}W(i)P_{1}(i) \leq \sum_{i=1}^{K}W(i)P_{2}(i)$.
Thus we have the desired result.
\end{proof}

\section{Internal probability theory}\label{sec:inp}

In this section, we give a brief introduction to nonstandard probability theory.
The interested reader can consult \citep{NSAA97} and \citep{Keisler87} for more details.

Consider a $\sigma$-algebra $\GSA$ on a space $X$, and the space $\ProbMeasures{X,\GSA}$ of countably additive probability measures defined on $(X,\GSA)$.
By the transfer principle, $\NSE{(X,\GSA)}$ and the set $\NSE{\!\!\!\ProbMeasures{X,\GSA}}$ of internal $\NSE{}$countably
additive probability measures on $\sGSA$
satisfy (the transfer of) all the first-order properties of their standard counterparts.
Some care is required: e.g., $\NSE{}$countably additivity is defined by the behavior of a measure on \emph{internal} sequences in $\sGSA$, not arbitrary sequences.  In general, the transfer principle is the primary
means of relating internal integration/measure theory to its standard counterpart. Saturation can then be used to control the effect of ``small'' perturbations.

As an example of using transfer, consider one of the key structures in this paper:
that of an indexed family of probability measures on a common measurable space $(X,\GSA)$.
Let $\Theta$ be an index set,
and, for every $\theta \in \Theta$, let $\Model_{\theta}$ be a probability measure on $(X,\GSA)$.
Equivalently, we can think of $\Model$ as a function
$\Model:\Theta \times \GSA \to [0,1]$ defined by $\Model(\theta,A)=\Model_{\theta}(A)$ for $A \in \GSA$.
By the transfer principle,
we know that $\NSE{\Model}$ is then a function from $\NSE{\Theta} \times \sGSA$ to $\NSE{[0,1]}$
and $\NSE{\Model}(y,\argdot)$ is an internal ($\NSE{}$countably additive) probability measure on $(\NSE{X},\sGSA)$ for every $y\in \NSE{\Theta}$.
Note that, for each $\theta\in \Theta$, we can also take the nonstandard extension of the probability measure $\Model_{\theta}$.
These two different nonstandard extensions agree with each other for $\theta\in \Theta$.

\begin{lemma}\label{nsfamilyresult}
For every $\theta\in \Theta$ and $A\in \sGSA$,
$\NSE{\Model}(\theta,A)= \NSE{(\Model_\theta)}(A)$.
\end{lemma}

\begin{proof} %
Fix any $\theta_{0}\in \Theta$ and define $F$ to be a function from $\GSA$ to $[0,1]$ by $F(B)=\Model(\theta_{0},B)$.
Thus $\NSE{F}$ is an internal function from $\sGSA$ to $\NSE{[0,1]}$ given by $\NSE{\Model}(\theta_0, A)$ for all $A\in \sGSA$.
Consider the sentence $(\forall B\in \GSA)(F(B)=\Model_{\theta_{0}}(B))$.
By the transfer principle, we have $(\forall A\in \sGSA) ( \NSE{F}(A)=\NSE{(\Model_{\theta_0})}(A))$.
As our choice of $\theta_0$ is arbitrary, we have $\NSE{\Model}(\theta,A)=\NSE{\Model}_{\theta}(A)$ for every $\theta\in \Theta$ and $A\in \sGSA$.
\end{proof}

Thus, for every $y\in \NSE{\Theta}$, we shall write $\NSE{P}_{y}(\argdot)$ for $\NSE{P(y,\argdot)}$ and keep in mind that we can view $\NSE{P}_{y}$ as the $y$-th fibre of a function of two variables. The next lemma also follows from a transfer  (and extension) argument.
\begin{lemma}
Let $(X,\GSA)$ be a measurable space, let $\{P_y\}_{y\in Y}$ be a family of probability measures on $(X,\GSA)$,  and
suppose $F$ is $\GSA$-measurable and $P_{y}$-integrable for all $y\in Y$.
Define $r(y) = \int_{X}F(x)P_{y}(\dee x)$ for $y \in Y$.
Then $\NSE{F}$ is $\sGSA$-$\NSE{}$measurable and $\NSE{P}_{y}$-$\NSE{}$integrable for all $y\in \NSE{Y}$,
\[
\NSE{r}(y) = \sint{\NSE{X}}\NSE{F}(x)\NSE{P_{y}}(\dee x)
\]
for all $y \in \NSE{Y}$, and $r(y) = (\NSE{r})(y)$ for every $y\in Y$.
\end{lemma}

We will simply write $\int$ for $\NSE{\!\!\int}$, integrable for $\NSE{}$integrable, etc., when the context is clear.  Saturation and transfer allow us to study the effects of ``small'' perturbations:
\begin{lemma}
\label{intinfint}
Let $(\gX, \GSA, P)$ be an internal probability space
and let $F,F'$ be internal $P$-integrable functions such that $F \approx F'$ everywhere.
Then $\int F \dee P \approx \int F' \dee P$.
\end{lemma}

\subsection{Hyperfinitely additive probability measures and Loeb theory}
Due to saturation, the subset of internal $\NSE{}$finitely additive probability measures defined on
internal algebras plays a central role in nonstandard probability theory.
We can understand $\NSE{}$finite (also called \defn{hyperfinite}) structures by the transfer principle. In particular,
a set $A\in \bV(\NSE{S})$ is hyperfinite if and only if there exists an internal bijection between $A$ and $\{1,2,\dotsc,N\}$ for some $N\in \NSE{\Naturals}$.
If such a number, $N$, exists, then it is unique and called the \defn{internal cardinality} of $A$.

By definition, hyperfinite sets are themselves internal (otherwise the bijection would not be internal).
By transfer, hyperfinite sets are well behaved like their standard finite counterparts:
e.g., hyperfinite sums and products are always convergent.
Internal subsets of hyperfinite sets are hyperfinite. The converse holds as well:
a subset of a hyperfinite set is internal if and only if it is hyperfinite \citep[][Exercise~1.6.17]{NSAA97}.

The following definitions align with the transfer principle:
An \defn{internal algebra} $\GSA \subset \PowerSet{X}$ is an internal set containing $X$ and closed under complementation and hyperfinite unions/intersections.
A set function $P\colon \GSA \to \NSE{\Reals}$ is \defn{hyperfinitely additive}
when, for every $n\in \NSE{\Naturals}$ and pairwise disjoint family $A_{1},\dotsc,A_{n} \in \GSA$,
we have $P(\bigcup_{i\leq n}A_i)=\sum_{i\leq n} P(A_i)$.
An \defn{internal (hyperfinitely additive) probability space} is a triple $(\Omega,\GSA,P)$
composed of
an internal set $\Omega$;
an internal subalgebra $\GSA \subset \PowerSet{\Omega}$;
and an internal hyperfinitely additive probability measure $P\colon \GSA \to \NSE{[0,1]}$ on $(\Omega,\GSA)$,
i.e., a nonnegative hyperfinitely additive internal function such that $P(\Omega)=1$ and $P(\emptyset)=0$.
A \defn{hyperfinite probability space} is an internal probability space $(\Omega,\GSA,P)$
such that $\Omega$ is a hyperfinite set
and $\GSA=\InternalSubsets{\Omega}$,
where $\InternalSubsets{\Omega}$ denotes the collection of all internal subsets of $\Omega$.
Like finite probability space, we can specify an internal probability measure on $\InternalSubsets{\Omega}$ by defining the mass of each $\omega\in \Omega$.

One of the key theorems in modern nonstandard measure theory is due to \citet{Loeb75},
who showed that any internal probability space can be extended to a standard $\sigma$-additive probability space.

\begin{theorem}[{\citep{Loeb75}}]\label{loeb75} Let $(\Omega,\GSA,P)$ be an internal finitely additive probability space.
Then there is a standard $\sigma$-additive probability space $(\Omega,\LoebAlgebraX{\GSA}{P},\Loeb{P})$ such that:

\begin{enumerate}

\item $\LoebAlgebra{\GSA} = \LoebAlgebraX{\GSA}{P}$ is a $\sigma$-algebra with $\GSA\subset \LoebAlgebra{\GSA} \subset \PowerSet{\Omega}$.

\item $\Loeb{P}(A)={\SP{P(A)}}$ for every $A\in \GSA$.

\item For every $A \in \LoebAlgebra{\GSA}$ and standard $\epsilon>0$, there exists $A_i,A_o\in \GSA$ such that $A_i\subset A\subset A_o$ and $P(A_o\setminus A_i)<\epsilon$.

\item For every $A\in \LoebAlgebra{\GSA}$, there exists $B\in \GSA$ such that $\Loeb{P}(A \symdiff B)=0$.

\end{enumerate}

\end{theorem}

The probability triple $(\Omega,\LoebAlgebra{\GSA},\Loeb{P})$ is called the \defn{Loeb space} of $(\Omega,\GSA,P)$. The $\sigma$-algebra $\Loeb{\GSA}$ and the probability measure $\Loeb{P}$ are called the Loeb extensions of $\GSA$ and $P$, respectively.
From Loeb's original proof, we can give the explicit form of $\Loeb{\GSA}$ and $\Loeb{P}$:

\begin{enumerate}

\item $A \in \Loeb{\GSA} \iff
\forall \epsilon\in \PosReals \, \exists A_i,A_o\in \GSA \, (A_i\subset A\subset A_o ) \land (P(A_o\setminus A_i)<\epsilon)$. \label{loebsets}

\item $(\forall A\in \Loeb{\GSA} ) \, \Loeb{P}(A)=\inf\{\Loeb{P}(A_o)|A\subset A_o \in \GSA\}=\sup\{\Loeb{P}(A_i)| A \supset A_i \in \GSA \}$.\label{loebmeasure}

\end{enumerate}

The next example demonstrates that $\Loeb{\GSA}$ may contain external sets.

\begin{example}[{\citep[][Exercise~2.2]{NSAA97}}]\label{representlebesgue}
Pick any $N\in \NSE{\Naturals} \setminus \Naturals$ and let $\delta t=\frac{1}{N}$. Then $\delta t$ is an infinitesimal.
Let $\Omega=\{0,\delta t,2 \delta t, \dotsc, 1\}$ and $\GSA=\InternalSubsets{\Omega}$.
Define $P$ on $\GSA$ by $P(\{\omega\})=\delta t$ for all $\omega\in \Omega$.
Then $(\Omega,\GSA,P)$ is a hyperfinite probability space. Let $(\Omega,\Loeb{\GSA},\Loeb{P})$ be the corresponding Loeb space, known as the
\defn{uniform hyperfinite Loeb space}.

\begin{claim}
$\mu(0)\cap \Omega\in \Loeb{\GSA}$.
\end{claim}

\begin{proof}
$\mu(0)\cap \Omega$ consists of elements from $\Omega$ that are infinitesimally close to $0$. For $n \in \Nats$, let $A_{n}=\{\omega\in \Omega: \omega\leq \frac{1}{n}\}$, which is internal by the internal definition principle.
Thus $A_{n} \in \GSA$ and
$\Loeb{\GSA} \ni \bigcap_{n\in \Naturals}A_{n} = \mu(0)\cap \Omega$, completing the proof.
\end{proof}

Note that $\mu(0)\cap \Omega$ is an external set. This shows that $\Loeb{\GSA}$ contains external sets.

In fact,
letting $\nu$ denote Lebesgue measure,
one can show that, for every set $A \subseteq [0,1]$ that is $\nu$-measurable,
$\ST^{-1}(A)\cap \Omega\in \Loeb{\GSA}$
and  $\nu(A)=\Loeb{P}(\ST^{-1}(A)\cap \Omega)$.
Thus $(\Omega,\GSA,P)$ is a ``hyperfinite representation'' of Lebesgue measure on $[0,1]$.
\end{example}

From Loeb's proof, we know that $\LoebAlgebraX{\GSA}{P}$ is
the $\Loeb{P}$-completion of the $\sigma$-algebra generated by $\GSA$.
This suggests that $\LoebAlgebraX{\GSA}{P}$ will depend on $P$. However, some sets always appear in the Loeb $\sigma$-algebra:

\begin{definition}\label{defunloeb}
A set $A\subset \Omega$ is called \defn{universally Loeb measurable} if $A\in \LoebAlgebraX{\GSA}{P}$ for every internal probability measure $P$ on $(\Omega,\GSA)$.
\end{definition}

We denote the collection of all universally Loeb measurable sets on an internal algebra $\GSA$ by $\ULoeb{\GSA}$.
The following theorem characterizes the universal Loeb measurability of the set of near-standard points and Borel sets under sufficient regularity conditions:

\begin{theorem}[{\citep[][Cor.~3]{LR87}}]\label{LRunivLoeb}
Let $\gY$ be a Hausdorff space with Borel $\sigma$-algebra $\BorelSets \gY$
and assume our nonstandard model is more saturated than $\aleph_0$ and the cardinality of the topology on $\gY$.
Then:
\begin{enumerate}
\item $\NS{\NSE{\gY}} \in \ULoeb{\NSE{\BorelSets \gY}}$ for locally compact spaces, for $\sigma$-compact spaces, and for complete metric spaces;
\item $\ST^{-1}(B) \in \{A \cap \NS{\NSE{\gY}}: A\in \ULoeb{\NSE{\BorelSets \gY}} \}$, $B \in \BorelSets{\gY}$, for regular spaces.
\end{enumerate}
\end{theorem}

\cref{LRunivLoeb} implies that, if $\gY$ is a $\sigma$-compact or locally compact Hausdorff space,
then $\ST^{-1}(B)\in \ULoeb{\GSA}$ for all $B\in \BorelSets \gY$.

\vfill

\end{document}